\documentclass[final,leqno]{siamltex}

\usepackage[algo2e,linesnumbered,vlined,ruled]{algorithm2e}
\usepackage{graphicx,graphics,subfigure}
\usepackage{hyperref,url}
\usepackage{epsf,epstopdf}

\usepackage{comment}
\usepackage{times,hyperref}
\usepackage{amsmath,amsxtra,amsfonts,amscd,amssymb,bm}
\usepackage[usenames]{color}
\usepackage[normalem]{ulem}
\usepackage{exscale}

\usepackage{longtable,booktabs}
\usepackage{float}

\newtheorem{assumption}[theorem]{Assumption}

\newcommand{\be}{\begin{equation}}
    \newcommand{\ee}{\end{equation}}
\newcommand{\bee}{\begin{equation*}}
    \newcommand{\eee}{\end{equation*}}
\newcommand{\bea}{\begin{eqnarray}}
    \newcommand{\eea}{\end{eqnarray}}
\newcommand{\beaa}{\begin{eqnarray*}}
    \newcommand{\eeaa}{\end{eqnarray*}}

\newcommand{\st}{\,\textrm{s.t.}\,}  
\newcommand{\R}{\mathbb{R}}  
\newcommand{\C}{\mathbb{C}}
  
\newcommand{\bba}{\mathbf{a}}  
\newcommand{\bbH}{\mathbf{H}}  
\newcommand{\cA}{\mathcal{A}}  
 
\newcommand{\cD}{\mathcal{D}}

\newcommand{\cI}{\mathcal{I}}  
\newcommand{\cJ}{\mathcal{J}}  
  
\newcommand{\cM}{\mathcal{M}}  

\newcommand{\cT}{\mathcal{T}}  
\newcommand{\cO}{{\mathcal{O}}}

\newcommand{\tr}{\mathrm{tr}}

\newcommand{\prox}{\mathbf{prox}} 
\newcommand{\iprod}[2]{\left\langle{#1},{#2}\right\rangle}

\newcommand{\ivec}{\mathrm{vec}}
\newcommand{\svec}{\mathrm{svec}}
\newcommand{\opt}{\mathrm{opt}}

\newcommand{\mol}[1]{\phi_{#1}}

\newcommand{\bes}[1]{\begin{equation} \left\{ \begin{split}#1\end{split} \right. \end{equation}}

\DeclareMathOperator*{\argmin}{arg\,min}

\allowdisplaybreaks

\graphicspath{ {./figure/} }

\title{A Semi-smooth Newton Method For Solving semidefinite programs in
electronic structure calculations}
\author{ Yongfeng Li \thanks{Beijing International Center for Mathematical Research, Peking University, CHINA (YongfengLi@pku.edu.cn).}
 \and
  Zaiwen Wen\thanks{Beijing International Center for Mathematical
Research, Peking University, CHINA (wenzw@pku.edu.cn).
Research supported in part by NSFC grant 91330202, and by the National Basic Research Project under the grant 2015CB856002.}
\and 
Chao Yang \thanks{Computational Research
Division, Lawrence Berkeley National Laboratory, Berkeley, UNITED
STATES (cyang@lbl.gov).  Support for this work was provided through 
the Scientific Discovery through Advanced Computing (SciDAC) program funded 
by U.S. Department of Energy, Office of Science, Advanced Scientific Computing 
Research (and Basic Energy Sciences), and by the Center for Applied 
Mathematics for Energy Research Applications (CAMERA)
under award number DE-SC0008666, 
}
\and
  Yaxiang Yuan\thanks{State Key Laboratory of Scientific and Engineering Computing, Academy of Mathematics and Systems Science, Chinese Academy of Sciences,
China (yyx@lsec.cc.ac.cn). Research supported in part by NSFC grants 11331012 and 11461161005.}
}

\begin{document}
\maketitle

\begin{abstract}
The ground state energy of a many-electron system can be approximated by an variational approach in which the total energy of the system is minimized with respect to one and two-body reduced density matrices (RDM) instead of many-electron wavefunctions. This problem can be formulated as a semidefinite programming problem. Due the large size of the problem, the well-known interior point method can only be used to tackle problems with a few atoms.  First-order methods such as the the alternating direction method of multipliers (ADMM) have much lower computational cost per iteration. However, their convergence can be slow, especially for obtaining highly accurate approximations.  In this paper, we present a practical and efficient second-order semi-smooth Newton type method for solving the SDP formulation of the energy minimization problem. We discuss a number of techniques that can be used to improve the computational efficiency of the method and achieve global convergence.  Extensive numerical experiments show that our approach is competitive to the state-of-the-art methods in terms of both accuracy and speed.
\end{abstract} 

%=================================Introduction====================================

\section{Introduction}
The molecular Schr\"odinger's equation, which is a many-body eigenvalue problem,
is a fundamental problem to solve in quantum chemistry. Because the 
eigenfunction to be determined is a function of $3N$ spatial variables,
where $N$ is the number of electrons in a molecule, a brute force approach
to solving this equation is prohibitively costly. The most commonly used
approaches to obtaining an approximate solution, such as the configuration
interaction~\cite{ci} and coupled cluster methods~\cite{cc}, express the approximate eigenfunction
as a linear or nonlinear combination of a set of 
many-body basis functions (i.e. Slater determinants), and determine the 
expansion coefficients 
by solving a projected eigenvalue problem or a set of nonlinear equations.
One has to choose the set of many-body basis functions judiciously to balance
the computational cost and the accuracy of the approximation. To 
reach chemical accuracy, the number of basis functions can still grow
rapidly with respect to $N$.

An alternative way to approximate the ground state energy (i.e., the smallest
eigenvalue), which does not involve approximating the many-body 
eigenfunction directly, is to reformulate the 
problem as a convex optimization problem and express the ground state energy 
in terms of the so called one-body reduced density matrix (1-RDM) and 
two-body reduced density matrix (2-RDM) that satisfy a number of linear 
constraints.
This convex optimization problem is a semidefinite program (SDP) that
can be solved by a number of numerical algorithms to be presented below.
This approach is often referred to as the variational 2-RDM (v2-RDM) or 
2-RDM method in short.

The development of the 2-RDM method dates back to 1950s. 
Mayer \cite{mayer1955electron} showed how the energy of a many-body 
problem can be represented in terms of 1-RDM and 2-RDM, which can be
write as a matrix and a 4-order tensor.
However, since not all matrices or tensors are RDMs associated with 
an $N$-electron wavefunction, one must add some constraints 
to guarantee that the matrices and tensors satisfy the so called 
\textit{N-representability condition}, which was first proposed by Coleman
\cite{A.J.Coleman1963} in 1963 and has been investigated for nearly 50 years. 
The N-representability condition for the 1-RDM in the variational problem 
has been solved in~\cite{A.J.Coleman1963}.  In 1964, Garrod and Percus~\cite{Garrod1964} 
showed a sufficient and necessary condition for the 2-RDM N-representability 
problem. It is theoretically meaningful but computationally intractable. 
In 2007, Liu et al. showed that the
N-representability problem of 2-RDM is QMA-complete~\cite{Liu2007}. Since then
a number of approximation conditions, including the P, Q, R, T1, T2, T2' 
conditions, have been proposed in \cite{A.J.Coleman1963,
Garrod1964, erdahl1978representability,  Zhao2004, mazziotti2006variational, braams2007t1}. All these conditions are formulated by keeping matrices whose elements are linear combinations of the components of 
the 1-RDM and 2-RDM matrices positive semidefinite. As a result, the constrained minimization of 
the total energy with respect to 1-RDM and 2-RDM becomes a SDP.

The practical use  of the v2-RDM approach to solving the 
ground state electronic structure is enabled, to some extent, by
the recent advances in numerical methods for solving large-scale SDPs.
In \cite{Nakata2001}, Nakata et al. solved the v2-RDM problem by an interior
point method.  Zhao et al. reformulated the 2-RDM using the 
dual SDP formalism and also applied the interior point method in
\cite{Zhao2004}. The problem size of the SDP formulation 
 in \cite{Zhao2004} is usually smaller than
the ones given in \cite{Nakata2001}.  
Rigorous error bounds for approximate solutions obtained from the v2-RDM 
approach are discussed in \cite{Chaykin2016}. 
%by considering all numerical errors.
Since the computational cost of the interior point method is typically high, this approach has only been 
successfully used for a handful of small molecules with a few atoms.
First order methods, which have much lower complexity per iteration,
have gained wide acceptance in recent years. The well-known alternating
 direction multiplier method (ADMM) has been used to solve
general SDPs in~\cite{Wen2010a}. It is the basis of the boundary point method
developed by Mazziotti to solve the v2-RDM in~\cite{Mazziotti2011}. 
Although ADMM has relatively low complexity per iteration,
it may converge slowly and take thousands or tens of thousands iterations 
to reach high accuracy. Recently, some new methods have been developed
to speed up the solution of general SDPs.
An example is the Newton-CG Augmented Lagrangian Method for SDP (SDPNAL) 
proposed in~\cite{Zhao2009}. %It shows a  faster convergent rate than ADMM but processes a lower cost than the interior point.
An enhanced version of SDPNAL called SDPNAL+ is developed 
in~\cite{Yang2015}, which can efficiently treat nonnegative SDP matrices.
However, these methods have not been applied to the v2-RDM 
approach for electronic structure calculation.

In this paper, we first review how the ADMM method is used to solve the 
SDP formulation of the v2-RDM given in~\cite{Zhao2004}
since it serves as the foundation of the second-order method to 
be introduced below. 
%Numerical experiments show that it is able to achieve a moderate  accuracy quickly but may take thousands of iterations to for a slightly higher accuracy. This usually restricts the use of ADMM method in many interesting applications. Fortunately, a few properties of ADMM enables us to overcome its drawbacks. 
We point out a key observation that applying the ADMM to the dual SDP formulation
is equivalent to applying the Douglas Rachford splitting (DRS)~\cite{DR1956,LM1979,EB1992} method to the primal SDP formulation of the problem. The DRS method 
can be viewed as a fixed point iteration that yields a solution
of a system of semi-smooth and monotone nonlinear equations that
coincides with the solution of the corresponding SDP.
The generalized Jacobian of this system of nonlinear equations is positive 
semidefinite and bounded. It has a special structure that allows us 
to compute the Newton step in our semi-smooth Newton method 
for solving SDPs efficiently.  We apply the semi-smooth Newton method
to the v2-RDM formulation of the ground state energy minimization problem,
and use a hyperplane projection technique~\cite{SS1999} to guarantee the 
global convergence of the method due to the monotonicity of the system of 
nonlinear equations. Our method is different from SDPNAL \cite{Zhao2009} 
and SDPNAL+ \cite{Yang2015} which minimizes a sequence of augmented 
Lagrangian functions for the dual SDP by a semi-smooth Newton-CG method. 
To improve the computational efficiency for solving v2-RDM problem, 
we exploit the special structures of matrices resulting from the 1-RDM and 
2-RDM constraints.  The block diagonal and low rank structures of these 
matrices are related to spin and spatial symmetry of the molecular 
orbitals~\cite{Gidofalvi2005,Nakata2001, Zhao2004}.  We show how 
they can be used to significantly reduce the computational costs in the 
semi-smooth Newton method.  Finally, we implement our codes based on the key implementation 
details and subroutines of SDPNAL \cite{Zhao2009}, SDPNAL+ \cite{Chaykin2016} and ADMM+ \cite{MR3342702}.
Extensive numerical experiments on examples taken from~\cite{Nakata2008} to demonstrate that 
our semi-smooth algorithm can indeed achieve higher accuracy than the ADMM 
method. We also show that it is competitive with SDPNAL and SDPNAL+ in terms of 
both computational time and accuracy.

The rest of this paper is organized as follows. In section \ref{sec:model},
we provide some background on electronic structure calculation, 
establish the notation and introduce the v2-RDM formulation. 
In section \ref{sec:admm},  we review first order methods suitable for
solving the SDP problem arising in the v2-RDM formulation. In particular, we 
examine the relationship between the ADMM and the DRS. We present a 
semi-smooth Newton method for solving the v2-RDM in section \ref{sec:nt}.  
Numerical results 
are reported in section \ref{sec:num}. Finally, we conclude the paper in section
\ref{sec:con}. 
%we elaborate on numerical experiments and show that our algorithm is efficient for 2-RDM problem and competitive to SDPNAL in time and accuracy. The effect of structure information is also shown in this section. %Section \ref{sec:con} concludes the paper.

%=================================Model Problem====================================

\section{Background} \label{sec:model}
\subsection{The variational 2-RDM formulation of the electronic structure problem}
The electronic structure of a molecule can be determined by the solution to 
an $N$-electron Schr\"odinger equation
\be\label{eq:sch}
H \Psi = E \Psi,
\ee
where $\Psi: \R^{3N} \otimes \{\pm \frac{1}{2}\}^{3N} \rightarrow \C$ is a
N-electron antisymmetric wave function that obeys the Pauli exclusion 
principle, $E$ represent the total energy of the $N$-electron system, and
$H$ is the molecular Hamiltonian operator defined by
\be\label{eq:Hsch}
H = \underbrace{\sum_{i=1}^N -\frac{1}{2}\triangle_i - \sum_{i=1}^N \sum_{k = 1}^K \frac{Z_k}{|R_k-r_i|}}_{\text{one-body term}} + \underbrace{\frac{1}{2}\sum_{i,j = 1,\\ i\neq j}^N \frac{1}{|r_i - r_j|}}_{\text{two-body term}}.
\ee
Here $\triangle_i$ denotes a Laplace operator with respect to the 
spatial coordinate of the $i$-th electrons, $R_k, k = 1,\cdots,K$,
gives the coordinates of the $k$-th nuclei with charge $Z_k$, and 
$r_i, i = 1,\cdots,N$, gives the coordinates of the $i$-th electron. 

To simplify notation, let us ignore the spin degree of freedom.  
In this case, the wave function $\Psi$ belongs to the the Hilbert 
space $L_2((\R^3)^N)$ endowed with the inner product
\[
\iprod{\Psi_1}{\Psi_2} =  \sum_{s=\pm\frac{1}{2}} \int_{\R^{3N}} \overline{\Psi_1(r_1,\cdots,r_N)}\Psi_2(r_1,\cdots,r_N)dr_1\cdots dr_N.
\]
The smallest eigenvalue of $H$, often denoted by $E_0$, is called the 
ground state energy \eqref{eq:sch}. 

Solving \eqref{eq:sch} directly is not computationally feasible except 
for $N=1$ or $N=2$. A commonly used approach in quantum chemistry is 
to approximate $\Psi$ from a \textit{configuration interaction} subspace
spanned by a set of many-body basis function $\Phi_i$, often chosen to be
\textit{Slater determinants} of the form
\begin{equation} 
  \Phi_i(r_1, r_2, \hdots, r_{N}) = \frac{1}{\sqrt{N!}} \begin{vmatrix} \mol{i_1}(r_1) & \mol{i_2}(r_1) & \hdots & \mol{i_{N}}(r_1) \\ \mol{i_1}(r_2) & \mol{i_2}(r_2) & \hdots & \mol{i_{N}}(r_2) \\ \vdots & \vdots &  & \vdots \\ \mol{i_1}(r_{N}) & \mol{i_2}(r_{N}) & \hdots & \mol{i_{N}}(r_{N})\end{vmatrix}, 
\label{eq:slater}
\end{equation}
where $\{\mol{i}(r\}$ is a set of orthonormal basis functions known as
molecular \textit{orbitals}~\cite{SaaCheSho1003}. These orbitals can be obtained by 
substituting \eqref{eq:slater} into \eqref{eq:sch} and solving a nonlinear
eigenvalue problem known as the Hartree-Fock (HF) equation. 
The $N$ eigenfunctions associated with the smallest $N$ eigenvalues are known 
as the occupied HF orbitals. All other eigenfunctions are called 
\textit{unoccupied} or \textit{virtual} orbitals. The Slater determinant that 
consists of the $N$ occupied HF orbitals is called the HF Slater determinant,
and denoted by $\Phi_0$.

A new Slater determinant can be generated from an existing
Slater determinant by replacing one or more orbitals 
with others.  This process is often conveniently expressed 
through the use of creation and annihilation operators denoted by 
$\bba_i^{+}$ and $\bba_i$ respectively~\cite{Szalay2015}. 
The successive applications of different combinations of creation and 
annihilation operators to the HF Slater determinant that
replace occupied orbitals with unoccupied orbitals allow us to
generate a set of Slater determinants that can be used to expand an 
approximate solution to~(\ref{eq:sch}). The entire set of 
such Slater determinants defines the so called \textit{full configuration
interaction} (FCI) space.  The FCI approximation to the solution of 
\eqref{eq:sch} is often used as the baseline for assessing the 
accuracy of approximate solutions to \eqref{eq:sch}. The size 
of the FCI space depends on the number of electrons $N$ and 
the number of degrees of freedom ($d$) used to discretize each
orbital $\phi_i$ (i.e., the basis set size in the quantum
chemistry language). When $N$ is large and an accurate basis set is 
used to discretize $\phi_i$, the FCI space can be extremely large.
Hence, FCI calculation can only be performed for small molecules 
in a small basis set.
%Since a rigorous description of the model is tedious, we only provide a rather informal statement in this paper.  
%Let $\{\phi_i(r,s)\}_{i=1}^d$ denote a set of orthonormal basis
%functions of $L_2(\R^3\otimes{\pm\frac{1}{2}})$, which are usually called spin orbitals. $\cV$ denotes the finitely dimensional linear space of $L_2((\R^3\otimes{\pm\frac{1}{2}})^N)$ that is the span of all Slater determinants generated by $\{\phi_i(r,s)\}_{i=1}^d$.
%Let $\bba_i^+$ denote the
%creation operators which represents one electron creates on the $i$-th spin orbital.
% Let $ \bba_i$ denote the annihilation operators which describes one electrion
% annihilate on the $i$-th spin orbital. Then the restriction on $\cV$ of the Hamiltonian
% \eqref{eq:Hsch}

The matrix representation of the many-body Hamiltonian \eqref{eq:Hsch} 
in the space of Slater determinants is determined by one electron integrals 
\[
T_{i,j} = \int_{\R^3} \overline{\phi_i(r)}(-\frac{1}{2}\triangle-\sum_{c=1}^K \frac{Z_c}{|r-R_c|}) \phi_j(r)dr,
\]
and two electron integrals 
\[
V_{ij,kl} = \frac{1}{2}\int_{\R^3} \int_{\R^3} \overline{\phi_i(r)\phi_j(r')}\frac{1}{|r'-r|}\phi_k(r)\phi_l(r')drdr'.
\]
These integrals can be used to express the many-body Hamiltonian 
\eqref{eq:Hsch} using the so called second quantization notation:
\be\label{eq:hamsec0}
\mathbf{H} = \sum_{i,j}^d T_{i,j}\bba_i^+\bba_j + \sum_{i,j,k,l=1}^d V_{ij,kl} \bba_i^+  \bba_j^+ \bba_l \bba_k.
\ee

It is well known that the smallest eigenvalue of $\bbH$ can be obtained
from  the Rayleigh-Ritz variational principle via the solution of the 
following constrained minimization problem:
\be
%E_0 = \min \bra{\Psi} \bbH \ket{\Psi} \st \braket{\Psi}{\Psi} = 1.
\label{eq:minE}
E_0 = \min \iprod{\Psi}{ \bbH \Psi} \st \iprod{\Psi}{\Psi} = 1.
\ee
%We can rewrite the optimization problem as
%\be\label{eq:orisdp}
%\min \tr( \bbH D) \ \st \ \tr(D)= 1, D \succeq 0, \rank{D} = 1,
%\ee
%where 
%%$D = \Psi(z)\Psi^*(z')$ is a full density matrix and $z, z'$ represent all arguments. 
%$D = \Psi \Psi^*$ is a full density matrix and 
%$\tr(\cdot)$ is the trace operator. % defined by $\tr(X(z, z')) = \int X(z,z) dz$, 
% %and the rank and positive semidefinite constraints are defined under the basis of $\cV$.  %\CY{do we want to introduce SDP here?}
%Removing the rank constraint
% in \eqref{eq:orisdp} yields a SDP. The solution set of the SDP relaxation
% problem is the convex hull of  the solution set of \eqref{eq:orisdp}. %Since the energy and other useful properties are all the linear function of $D$, the solution of this relaxation SDP can be used to get the same value. 
% Since the matrix dimension of this relaxed SDP is $O(d^N)$, it is not suitable
% for practical computation. 
% %\CY{We need a transition from full $D$ to 1-RDM and 2-RDM.}
%However,
 When $\Psi$ is expanded in terms of Slater determinants, substituting
\eqref{eq:hamsec0} into \eqref{eq:minE} yields
\bea
E &=& %\bra{\Psi} \bbH \ket{\Psi}
\iprod{\Psi}{ \bbH \Psi} % \\ \nonumber
%&=& \sum_{i,j}^d T_{i,j}\iprod{\Psi} {\bba_i^+\bba_j \Psi} +\frac{1}{2}
%\sum_{i,j,k,l=1}^d V_{ij,kl}\iprod{\Psi}{\bba_i^+  \bba_j^+ \bba_l \bba_k \Psi} \\ \nonumber
 = \sum_{i,j}^d T_{i,j}\gamma_{i,j} + \sum_{i,j,k,l=1}^d V_{ij,kl}\Gamma_{ij,kl}, \label{eq:erdm}
\eea
where 
\be
\gamma_{i,j} = \iprod{\Psi} {\bba_i^+\bba_j \Psi}
\ \ \mbox{and} \ \
\Gamma_{ij,kl} = \iprod{\Psi}{\bba_i^+  \bba_j^+ \bba_l \bba_k \Psi}
\label{eq:rdms}
\ee
are elements of the so-called one-body reduced density matrix (1-RDM) $\gamma$ 
and two-body reduced density matrix (2-RDM) $\Gamma$, respectively.  

Note that the dimensions of $\gamma$ and $\Gamma$ are $d\times d$ and 
$d^2\times d^2$ respectively, where $d$ is proportional to the number
of electrons $N$. By treating 
the total energy $E$ as a function of $\gamma$ and $\Gamma$, we can 
obtain an approximation to the ground state energy by solving an optimization
problem with $O(N^4)$ variables instead of an eigenvalue problem of 
a dimension that grows exponentially with respect to $N$.

However, $\gamma$ and $\Gamma$ are not arbitrary matrices. They 
are said to be $N$-representible if they can be written as \eqref{eq:rdms}
for some many-body wavefunction $\Psi$. $N$-representible matrices 
are known to have a number of properties~\cite{mazziotti2006variational,Zhao2004} that can be used to
constrain the set of matrices over which the objective function 
\eqref{eq:erdm} is minimized. These properties include
%\begin{itemize}
%\item Hermitian: \be \gamma_{i,j} = \gamma_{j,i}, \Gamma_{ij,kl}=
%  \Gamma_{kl,ij};\ee
%\item antisymmetric:  $\Gamma_{ij,kl} = -\Gamma_{ji,kl} = -\Gamma_{ij,lk}$;
%\item trace: 
%\be
%\tr(\gamma) = N \ \ \mbox{and} \ \  \tr(\Gamma) = \frac{N(N-1)}{2};
%\label{eq:rdmtrace}
%\ee
%\item a partial trace condition between the 1-RDM and 2-RDM: 
%  \be \sum_{k}\Gamma_{ik,jk} = \frac{N-1}{2}\gamma_{ij}.\label{eq:rdmtrace2}
%\ee
%\end{itemize}
\bea
\gamma_{i,j} = \gamma_{j,i}, \Gamma_{ij,kl} &=&
  \Gamma_{kl,ij}; \quad \quad \quad\quad \mbox{Hermitian} \label{eq:rdmherm}
\\
\Gamma_{ij,kl} = -\Gamma_{ji,kl} &=& -\Gamma_{ij,lk}; \label{eq:rdmanti}
 \quad \quad \quad \; \mbox{anti-symmetric }\\
\tr(\gamma) = N \ \ \mbox{and} \ \  \tr(\Gamma) &=& \frac{N(N-1)}{2};\label{eq:rdmtrace}
 \quad \; \; \mbox{trace } \\
  \sum_{k}\Gamma_{ik,jk} &=& \frac{N-1}{2}\gamma_{ij}. \quad \quad \mbox{ partial
  trace}\label{eq:rdmtrace2} 
\eea

However, the above conditions are not sufficient to guarantee 
$\gamma$ and $\Gamma$  to be $N$-representible. A significant amount of
effort has been devoted in the last few decades to develop
additional conditions that further constrain $\gamma$ and $\Gamma$
to be $N$-representible~\cite{mazziotti2006variational,Zhao2004} without making use of 
$\Psi$ explicitly. These conditions are 
collectively called the \textit{N-representability conditions}. 

\subsection{N-representability conditions}
 The N-representability conditions were first introduced in
 \cite{A.J.Coleman1963}. 
It has been shown in~\cite{A.J.Coleman1963} that $\gamma$ is 
N-representable if and only if $0 \preceq \gamma \preceq I$.
For 2-RDM, it is more difficult to 
write down a complete set of the conditions under which $\Gamma$ is 
$N$-representable.  Liu et al. showed that the N-representability problem is QMA-complete %\CY{need to explain the acronym: no explanation in the original paper} 
in~\cite{Liu2007}. There has been efforts to derive 
 approximation conditions that are useful in practice. 
%The well known approximation conditions include the P, Q,
% R, T1, T2, T2' et al. conditions proposed in \cite{A.J.Coleman1963, Garrod1964,
% erdahl1978representability,  Zhao2004, mazziotti2006variational, braams2007t1}. They map $\gamma$ and $\Gamma$ to a set of positive semidefinite matrices 
% by linear operations, i.e., $\mathcal{L}(\gamma,\Gamma) \succeq 0$, 
% where $\mathcal{L}$ is a linear operator. The precise definitions of $\mathcal{L}$ are as follows:
%%\CY{what does $\mathcal{L}$ really refer to? one of the following or all of the following?}
The well known approximation conditions in \cite{A.J.Coleman1963, Garrod1964,
 erdahl1978representability,  Zhao2004, mazziotti2006variational, braams2007t1}
 define the so-called  $P, Q,
 R, T1, T2$ variables whose elements can be expressed as a linear function
 with respect to the elements of $\gamma$ and $\Gamma$ as follows: %et al. conditions proposed. They map $\gamma$ and $\Gamma$ to a set of positive semidefinite matrices by linear operations, i.e., $\mathcal{L}(\gamma,\Gamma) \succeq 0$, where $\mathcal{L}$ is a linear operator. The precise definitions of $\mathcal{L}$ are as follows: 
 \bea  \label{eq:p}
 P_{ij,i'j'} &=&  \iprod{\Psi} {\bba_i^+  \bba_j^+ \bba_j' \bba_i' \Psi} = \Gamma_{ij,i'j'}, \\
  Q_{ij,i'j'} &=& \iprod{\Psi} {\bba_i  \bba_j \bba_{j'}^+ \bba_{i'}^+ \Psi} 
    = (\delta_{ii'}\delta_{jj'} - \delta_{ij'}\delta_{ji'}) 
     - (\delta_{ii'}\gamma_{jj'}+\delta_{jj'}\gamma_{ii'}) \\ \nonumber
     &+& (\delta_{ij'}\gamma_{ji'}+\delta_{ji'}\gamma_{ij'}) + \Gamma_{ij,i'j'},\\
  G_{ij,i'j'} &=& \iprod{\Psi} {\bba_i^+  \bba_j \bba_{j'}^+ \bba_{i'} \Psi} = \delta_{jj'}\gamma_{ii'}-\Gamma_{ij',i'j} \\ {T1}_{ijk,i'j'k'} 
  &=& \iprod{\Psi} {(\bba_i^+  \bba_j^+ \bba_{k}^+ \bba_{k'} \bba_{j'} \bba_{i'}
  + \bba_i  \bba_j \bba_{k} \bba_{k'}^+ \bba_{j'}^+ \bba_{i'}^+) \Psi},\nonumber \\ 
    &=& \mathcal{A}[ijk]\mathcal{A}[i'j'k'](\frac{1}{6}\delta_{ii'}\delta_{jj'}\delta_{kk'} - \frac{1}{2}\delta_{ii'}\delta_{jj'}\gamma_{k,k'}+\frac{1}{4}\delta_{ii'}\Gamma_{jk,j'k'}),\\
 {T2}_{ijk,i'j'k'} &=& \iprod{\Psi} {(\bba_i^+  \bba_j^+ \bba_{k} \bba_{k'}^+
 \bba_{j'} \bba_{i'} + \bba_i^+  \bba_j \bba_{k} \bba_{k'}^+ \bba_{j'}^+
 \bba_{i'}) \Psi}\nonumber \\ 
&=& \mathcal{A}[jk]\mathcal{A}[j'k'](\frac{1}{2}\delta_{jj'}\delta_{k,k'}\gamma_{ii'} + \frac{1}{4}\delta_{ii'}\Gamma_{j'k',jk}-
    \delta_{jj'}\Gamma_{ik',i'k}), \label{eq:t2}
\eea
where  $\delta$ is the Kronecker delta symbol and $\cA[ijk]f(i,j,k) =
f(i,j,k)+f(j,k,i)+f(k,i,j)-f(i,k,j)-f(j,i,k)-f(k,j,i)$. The $T2$ variable can be
strengthened to yield the $T2'$ variable described
in~\cite{braams2007t1,mazziotti2006variational}. We should point out that each of
\eqref{eq:p}-\eqref{eq:t2} is in fact a set of equations enumerating all
possible indices $i,j,k,i',j'$ and $k'$.
 Since $\Gamma$ is a  $4$-dimensional tensor satisfying \eqref{eq:rdmherm}
 and \eqref{eq:rdmanti}, one
 can convert it to a
 two-dimensional matrix $\tilde \Gamma$, i.e., 
 \[\Gamma_{ij,i'j'}=\tilde \Gamma_{j - i + (2d - i)(i - 1)/2, j' - i' + (2d -
 i')(i' - 1)/2}.\]
 Similar properties hold for $Q$. Hence, $\Gamma$ and  $Q$ can be transformed into 
$\frac{d(d-1)}{2} \times \frac{d(d-1)}{2}$  matrices.
 Because \eqref{eq:rdmanti} is not satisfied on the $4$-dimensional tensor $G$,
 it can only be transformed into a $d^2\times d^2$  matrix.
 By the anti-symmetric properties  of the $6$-dimensional tensors $T1$, $T2$ and
 $T2'$ \cite{Zhao2004, braams2007t1,mazziotti2006variational},
 they can be
transformed into 
$\frac{d(d-1)(d-2)}{6} \times \frac{d(d-1)(d-2)}{6}$, $\frac{d^2(d-1)}{2} \times
\frac{d^2(d-1)}{2}$ and $\frac{d^2(d-1)+2d}{2} \times \frac{d^2(d-1)+2d}{2}$
 matrices, respectively. For simplicity, we still use the
 notations $\Gamma, P, Q, G, T1, T2$ and $T2'$ to represent the matrices translated from these tensors.
 Finally, the corresponding N-representability condition
 of \eqref{eq:p}-\eqref{eq:t2} is to require each matrix to be positive semidefinite.  

\subsection{The SDP formulations} \label{sec:sdp-2rdm}
Let $b = (\svec(T),\svec(V))^T \in \R^m $ and \\$y = (\svec(\gamma),\svec(\Gamma))^T \in \R^m $ be vectorized integral and reduced density 
matrices that appear in \eqref{eq:erdm} respectively, where $\svec$ is used 
to turn a symmetric matrix $U$ into a vector according to  
\[
\svec(U) = (U_{11},\sqrt{2}U_{12},U_{22},\sqrt{2}U_{13},\sqrt{2}U_{23},U_{33},\cdots,U_{nn}).
\]
To simplify notations later, we rename matrices as $S_1 = \gamma$, $S_2 =
P$, $S_3 = Q$, $S_4=G$, $S_5=T1$ and $S_6=T2$, and treat both $y$ and $\{S_j\}$
as variables in the SDP formulation. Using the definition of $y$, we can rewrite the equation $S_1 = \gamma$
as a system of linear equations
 \be S_1 = \cA_1^* y + C_1,\label{eq:con-gamma}\ee
where  $\cA_1^* y = \sum_{p=1}^m A_{1p} y_p\in \R^{s_1 \times s_1}$ with  $A_{1p} \in \R^{s_1 \times s_1}$ and 
$C_{1} \in \R^{s_1 \times s_1}$. Obviously, $s_1=d$ and $C_1$ is a zero matrix.
 Similarly, each of \eqref{eq:p}-\eqref{eq:t2}
can be written succinctly as 
\be 
S_j = \cA_j^* y - C_j, \; j=2, \ldots,l=6, \label{eq:con-S}\ee
where $\cA_j^* y = \sum_{p=1}^m A_{jp} y_p$ with  $A_{jp} \in \R^{s_j \times s_j}$ and 
$C_{j} \in \R^{s_j \times s_j}$.  The integer $s_j$ is equal to the matrix
size of $S_j$. The matrices $A_{jp}$
are coefficients matrices of $y_p$ and $C_j$ are constant matrices
in the corresponding equation of \eqref{eq:p}-\eqref{eq:t2}. %Of course, the value of $l$ % These coefficients are blocked and partitioned conformally with a partition of the vector $y$ (into $y_p$'s.) The matrices $C_j$ contain the constants (0's and 1's) that appear in \eqref{eq:p}-\eqref{eq:t2}.
%\CY{can we make sure dimensions matche? $\cA_j^* y$ is a vector? $C_j$ is a matrix?} The precise definition of $A_{jp} \in \R^{s_j \times s_j}$ and $C_{j} \in \R^{s_j \times s_j}$ can be found in~\cite{}. 

Using these notations, we can formulate the constrained minimization of
\eqref{eq:erdm} subject to $N$-representability conditions as a SDP:
\begin{equation}\label{eq:dualSDP}
	\begin{split}
        \min_{y, S_j} \ & b^Ty\\
		\st \ & S_j = \cA_j^* y - C_j, j = 1,\cdots,l,\\
		 & B^Ty = c, \\
		  & 0 \preceq S_1 \preceq I, \\
		 & S_j \succeq 0, j = 2,\cdots,l, \\
	\end{split}
\end{equation}
%where the detailed definitions of  %$b, y \in \R^m$, $\cA_j^* y = \sum_{p=1}^m A_{jp} y_p$, $A_{jp} \in \R^{s_j \times s_j}$, $C_{j} \in \R^{s_j \times s_j}$,  $A \in \R^{m\times s}$ and $c \in \R^s$ are omitted. 
where the linear constraints $B^Ty = c$ follows from the  conditions 
\eqref{eq:rdmtrace}-\eqref{eq:rdmtrace2} and other equality conditions introduced in~\cite{Zhao2004}. 
 %Other linear inequalities can be derived from \eqref{eq:p}-\eqref{eq:t2}. %\CY{not sure what is meant here?}
If some of conditions in \eqref{eq:p}-\eqref{eq:t2} are not considered,  then \eqref{eq:dualSDP} can be adjusted accordingly.  
 If the condition on $T2$ is replaced by
that of $T2'$, then we set 
$S_6=T2'$. %However, $T2$ and $T2'$ are not considered simultaneously. 

The SDP problem given in \eqref{eq:dualSDP} is often known as the
dual formulation.  The corresponding primal SDP of \eqref{eq:dualSDP} is 
\begin{equation}\label{eq:primalSDP}
	\begin{split}
        \max_{X_j,U} \ & \sum_{j=1}^l\iprod{C_j}{X_j} + \iprod{c}{x} - \iprod{C_1 + I}{U}\\
		\st \ & \sum_{j=1}^l\cA_j (X_j) + Bx - \cA_1(U) = b, \\
		& X_j \succeq 0, j = 1, \cdots, l, \\
		& U \succeq 0,
	\end{split}
\end{equation}
where $X_j \in \R^{s_j \times s_j}$, $U \in \R^{s_1 \times s_1}$, $\cA_j$ is the
conjugated operator of $\cA_j^*$ and $\cA_j(X) = (\iprod{A_{j1}}{X}, \cdots, \iprod{A_{jm}}{X} )^T$ for any matrix $X \in \R^{s_j \times s_j}$.

%\CY{Describe the number of primal and dual variables and the number of constraints to indicate that these are large-scale SDPs. }
Since the largest matrix dimension of $X_j$ and $S_j$ is of order $O(d^3)$ and
$m=O(d^4)$, problems \eqref{eq:dualSDP} and \eqref{eq:primalSDP} are large scale
SDPs even for a moderate value $d$. However, the $S_j$ in \eqref{eq:dualSDP} are block diagonal
 matrices  due to the spatial and spin symmetries of  molecules. Hence, 
the computational cost for solving \eqref{eq:dualSDP} 
can be reduced by exploiting such block diagonal structures.  
%By taking these symmetries into consideration, we can express
%$\gamma$, $\Gamma$, P, Q, R, T1, T2, T2' as block diagonal matrices 
%so that the computational cost of optimization can be reduced
%significantly. %The spin symmetry has been investigated in \cite{Zhao2004}, 
In Table~\ref{tab:blk}, we list the number of diagonal blocks and their
dimensions resulting 
from spin symmetries  
in each of $\gamma$, $\Gamma, Q, G, T1, T2, T2'$ matrices.
\begin{table}[!htb]
\caption{the matrix dimensions of the block diagonal structures}\label{tab:blk}
\begin{center}
\begin{tabular}{cc}
\hline
$S_j$ matrix & block dimension \\
	\hline
$\gamma$ & $\frac{d}{2}$, 2 blocks; \\
$P$, $Q$, $\Gamma$ & $\frac{d^2}{4}$, 1 blocks; $\frac{d}{4}(\frac{d}{2}-1)$, 2 blocks; \\
	$G $ & $\frac{d^2}{2}$, 1 blocks; $\frac{d^2}{4}$, 2 blocks;  \\
	$T1$ & $\frac{d^2}{8}(\frac{d}{2}-1)$, 2 blocks; $\frac{d^2}{12}(\frac{d}{2}-1)(\frac{d}{2}-2)$, 2 blocks;  \\
	$T2$ & $\frac{d^2}{8}(\frac{3d}{2}-1)$, 2 blocks; $\frac{d^2}{8}(\frac{d}{2}-1)$, 2 blocks;   \\
	$T2'$ & $\frac{d}{2}+\frac{d^2}{8}(\frac{3d}{2}-1)$, 2 blocks; $\frac{d^2}{8}(\frac{d}{2}-1)$, 2 blocks;   \\
	\hline
\end{tabular}
\end{center}
\end{table}

Spatial symmetry may lead to additional block diagonal structures
within each spin diagonal block listed in Table~\ref{tab:blk}.  
These block diagonal structures can be clearly seen within the 
largest spin block diagonal block of the $T_2$ matrices
associated with the carbon atom and the CH molecules shown in 
Figure~\ref{fig:sstr}. These $T_2$ matrices are generated from 
spin orbitals obtained from the solution of the HF equation discretized
by a double-$\zeta$ local atomic orbital basis.  The block diagonal
structure shown in Figure~\ref{fig:sstr} is obtained by applying 
a suitable symmetric permutation to the rows and columns of the 
$T_2$ matrices. 
 By representing the variables $S_j$ as block diagonal
 matrices whose sizes are much smaller, the off-diagonal parts of $S_j$ are no
 longer needed. Consequently, the length of $y$ may be reduced and 
 each of \eqref{eq:con-gamma}-\eqref{eq:con-S} may be split into several smaller
 systems.  Therefore, it is possible to generate a much
 smaller SDP. Without loss of generality, we still consider the formulation
 \eqref{eq:dualSDP} and our proposed algorithm can be applied to the reduced
 problems as well.%the size of the SDP problems can be significantly reduced. \CY{Need to say reduced from what to what?} 
%For a general basis,  one can use the coupled basis to generate the block diagonal structure \cite{Gidofalvi2005}.
\begin{figure}[!htb]
\centering
\subfigure[C atom]{
\includegraphics[width=0.38\textwidth,height=0.4\textwidth]{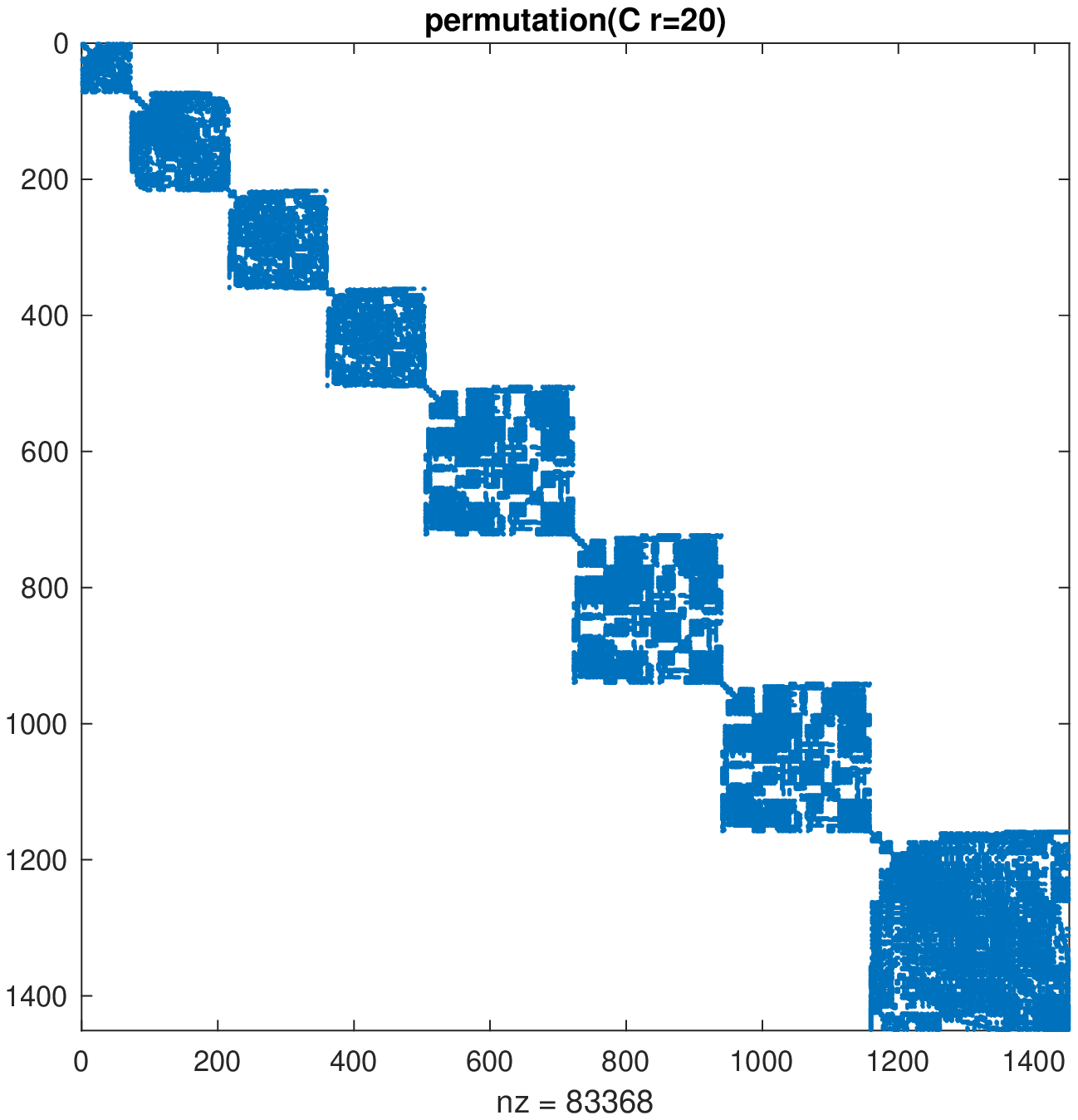}}
\subfigure[CH molecule]{
\includegraphics[width=0.53\textwidth,height=0.4\textwidth]{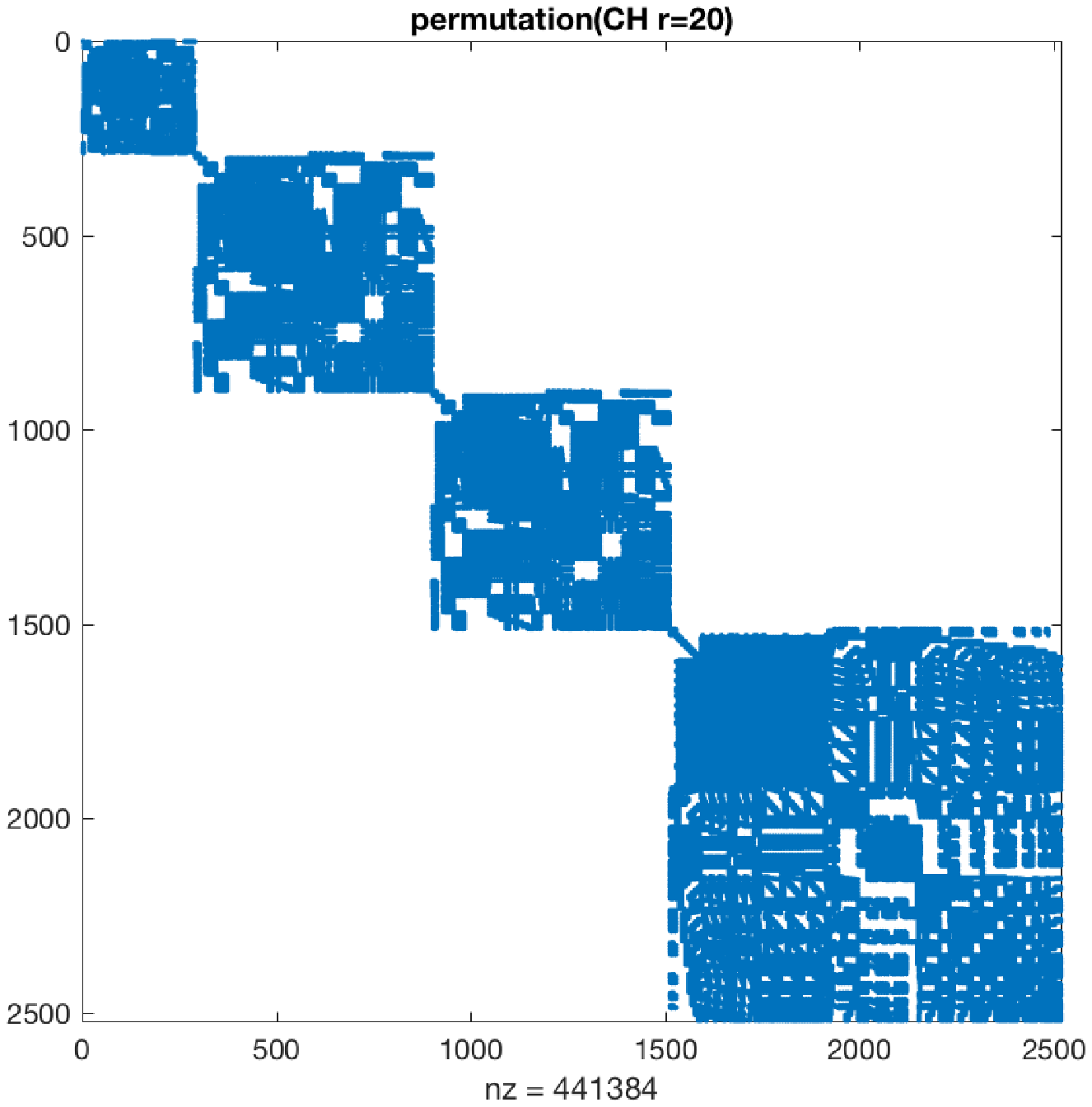}}
\caption{The block diagonal structures within the largest spin blocks of the 
$T_2$ matrices associated with the carbon (C) atom and the CH molecule.}\label{fig:sstr}
\end{figure}

In addition to exploiting the block diagonal structure in the 
$S_j$ matrices that appear in the dual SDP, we can also use the low rank 
structure of $\{X_i\}$ and $U$ to reduce the cost for solving 
\eqref{eq:primalSDP}.
The following theorem shows that $\{X_i\}$, $i=1,2,...,l$ and $U$ 
in the primal \eqref{eq:primalSDP} are indeed low rank as long as 
$d$ is sufficiently large.
\begin{theorem}\label{thm:lowrank}
  Assume that there exists matrices $\hat X_j\succ 0$ and
  $\hat U\succ 0$ such that
  the linear equality constraints of \eqref{eq:primalSDP} are satisfied 
  with them and the basis size $d$ is larger than $3$. 
Then there exists an optimal solution $\{X_1,\ldots,X_l,U\}$ of \eqref{eq:primalSDP} such that $r=\sum_{j=1}^l r_j + r_u\le \frac{\sqrt{3}}{8}(d^2+6)$, 
where $r_j$ is the rank of $X_j$ and $r_u$ is the rank of
$U$.  Moreover, $r_j / s_j = O(1/d)$ for $j$'s associated with
the T1, T2 and T2' conditions.
\end{theorem} 
\begin{proof}
We first prove that there must exist a solution such that $r \le \sqrt{m}$,
where $m$ is the length of the dual variable $y$ in \eqref{eq:dualSDP}. The
primal SDP \eqref{eq:primalSDP} can be written as a standard SDP in the
form of \eqref{eq:primSDP-s},
%\begin{equation}\label{eq:stdprimalSDP}
%    	\begin{split}
%    		\max \ & \iprod{C}{X},\\
%    		\st \ & \cA X = b, \\
%    		& X \succeq 0, 
%    	\end{split}
%\end{equation}
where $X$ is a block diagonal matrix whose diagonal parts are $U$, $X_j$ and $\diag(x)$. Then the size of $X$
is $\sum_{j=1}^ls_j + s_u + 2s$. Let the rank of $X$ be $\tilde{r}$. %For each solution of \eqref{eq:primSDP-s}, we can construct a solution of  the original SDP \eqref{eq:primalSDP} which total rank is less than the rank of the corresponding solution of \eqref{eq:stdprimalSDP}.
It follows from the results shown in~\cite{pataki1998rank} that
$\frac{\tilde{r}(\tilde{r}+1)}{2} \leq m$, which implies $\sum_{j=1}^ls_j + s_u
+ 2s \leq \sqrt{m}$.  Since
$m=\frac{3}{64}d^4-\frac{1}{16}d^3+\frac{9}{16}d^2+\frac{1}{4}d \le
(\frac{\sqrt{3}}{8}(d^2+6))^2$ when $d \geq 3$, the first statement holds. The
second statement follows from  Table \ref{tab:blk} that the dimension of the 
$S_j$ matrices associated with the T1, T2, T2' conditions are on 
the order of $O(d^3)$. 
%\CY{should probably explain where the formula for $m$
%comes from and why the T1,T2,T2' matrices are of dimension $O(d^3)$.} 
% wzw: m is the length of y
\end{proof}

%\CY{Do we talk about how the low rank structures are used in the solution of
%the SDP problem?}. wzw: it has been mentioned that this property will be useful in
%the computation later.

%=================================ADMM and DRS method====================================

\section{The ADMM and DRS method} \label{sec:admm}
We now discuss using first-order methods to solve the SDP formulations of 
the ground state energy minimization problem for a many-electron system.
%
%In this section, we first introduce the relationship between the ADMM and the
%DRS. It enables us to extend the  
% strategies for adjusting the parameters of the ADMM to the DRS. Then we 
%apply the ADMM to the 2-RDM.
%
For simplicity, let us first consider a generic SDP problem. 
Given $C,X \in \R^{n \times n}$, we define the linear operator 
$\mathcal{A}:\R^{n \times n} \rightarrow \R^m$ by
$
\cA X = (\iprod{A_1}{X},\cdots,\iprod{A_m}{X})^T
$
 where $A_1, \cdots, A_m \in \R^{n \times n}$. The conjugate operator of
 $\mathcal{A}$ is defined by $\cA^* y = \sum_{p=1}^m A_{p} y_p$ for $y \in \R^m$.
Using these notation, we can formulate a primal SDP as
\be\label{eq:primSDP-s}
\begin{aligned}&\max_{X} &\iprod{C}{X} \\
	&\st &\mathcal{A}X = b, \\ 
	&& X \succeq 0.
    \end{aligned}
\ee
The corresponding dual SDP is
\begin{equation}\label{eq:dualSDP-s}
	\begin{split}
        \min_{y,S} \ & b^Ty\\
		\st \ & S = \cA^* y - C,\\
		 & S \succeq 0.\\
	\end{split}
\end{equation}

\subsection{The DRS method}
The DRS method, first introduced to solve nonlinear partial differential 
equations~\cite{DR1956,LM1979,EB1992}, can be used to solve the 
primal SDP.  To describe the DRS method, we first establish some notations 
and terminologies.
Given a convex function $f$ and a scalar $t>0$, the {proximal mapping} of $f$ is defined by
\be\label{eq:prox}
\prox_{tf}(X):=\argmin\limits_{U} f(U)+\frac{1}{2t}\|U-X\|_F^2.
\ee 
We also define an indicator function on a convex set $\Omega$ as
\[
1_{\Omega}(X):=\begin{cases} 0, &\mbox{ if } X \in \Omega, \\
  +\infty, &\mbox{ otherwise}.\end{cases}
\] 
%The projection to a convex set $\Omega$ is 
%\[
%\cP_{\Omega}(X):=\argmin\limits_{U\in \Omega} \|U-X\|_F^2.
%\] 
To use the DRS method to solve \eqref{eq:primSDP-s}, we let 
\be f(X) = -\iprod{C}{X}  + 1_{\{\cA X =b\}}(X) \mbox{ and }
h(X)=1_{K}(X),\label{eq:def-fh}\ee
where $K = \{X: X \succeq 0\}$.
Then each iteration of the DRS procedure for solving \eqref{eq:primSDP-s} can 
be described by the following sequences of steps
\begin{equation}\label{eq:DRS-sd}
\begin{split}
&X^{k+1} = \prox_{th}(Z^k),\\
&U^{k+1} = \prox_{tf}(2X^{k+1}-Z^k),\\
&Z^{k+1} = Z^k+U^{k+1}-X^{k+1},
\end{split}
\end{equation} 
where $\{U^{k}\}$ and $\{Z^{k}\}$ are two sets of auxiliary variables.
It follows from some simple algebraic rearrangements that the variables $X$ and 
$U$ can be eliminated in \eqref{eq:DRS-sd} to yield a fixed-point iteration 
of the form 
\be\label{eq:fix-DRS}
Z^{k+1}=T_{\textrm{DRS}}(Z^k),
\ee
where
\be\label{eq:f-DRS}
T_{\textrm{DRS}}:=I+\prox_{tf}\circ(2\prox_{th}-I)-\prox_{th}.
\ee

\subsection{ADMM}
The ADMM is another method for solving the dual formulation of the
SDP ~\eqref{eq:dualSDP-s}.
Let $X$ be the Lagrangian multiplier associated with the linear equality
constraints of \eqref{eq:dualSDP-s}. The augmented Lagrangian function is
\[ L_{\mu}(y,S,X) = b^T y + \iprod{X}{S-\mathcal{A}^*y + C}
+\frac{\mu}{2}||S-\mathcal{A}^*y+ C ||_F^2.\]
Applying ADMM \cite{Wen2010a} to \eqref{eq:dualSDP-s} yields the 
following sequence of steps in the $k$th iteration
\begin{equation}\label{eq:ADMM-sd}
	\begin{split}
		&y^{k+1} = {\arg\min}_y L_{\mu}(y,S^k;X^k), \\
		&S^{k+1}  = {\arg\min}_{S \succeq 0} L_{\mu}(y^{k+1},S;X^{k}), \\
		&X^{k+1}  = X^k + \mu (S^{k+1} - \cA^* y^{k+1} + C). \\
	\end{split}
\end{equation}

In practice, the penalty parameter $\mu$ is often updated adaptively to
achieve faster convergence in the ADMM. One strategy is to tune $\mu$ to balance the primal infeasibility $\eta_p$ and the dual infeasibility $\eta_d$ defined by
\be\label{eq:pdinf}
\eta_p = \frac{\|\cA(X)-b\|_2}{\max(1,\|b\|_2)} \quad \text{and} \quad \eta_d= \frac{\|\cA^*y-C-S\|_F}{\max(1,\|C\|_F)}.
\ee
If the mean of $\eta_p/\eta_q$ in a few steps is larger (or smaller) than a
constant $\delta$, we decrease (or increase) the penalty parameter $\mu$ by
a multiplicative factor $\gamma$ (or $1/\gamma$) with $0 < \gamma < 1$. 
To prevent $\mu$ from becoming excessively large or small, 
a upper and lower bound are often imposed on $\mu$. This strategy has been 
demonstrated to be effectively in \cite{Wen2010a}. 
 %However, it does not work in the DRS directly since the variables $y$ and $S$ are not explicitly computed.\CY{this doesn't make sense since $\mu$ doesn't appear in DRS}  

\subsection{The connection between ADMM and DRS}
It is well known that the DRS for the primal \eqref{eq:primSDP-s} is equivalent
to the ADMM for dual \eqref{eq:dualSDP-s}. %If the two schemes start from a suitable point, they generate the same sequence. 
In particular, the $X$ variable produced in the $k$th step of DRS applied to
\eqref{eq:primSDP-s} is exactly the $X$ variable produced in the $k$th step 
of ADMM applied to \eqref{eq:dualSDP-s}. The other variables ($Z$ and $U$) and 
the parameter $t$ produced in DRS are related to the variables $y$, $S$ and parameter $\mu$ produced in the ADMM via
\bes{
t &= \mu; \\
Z^{k} & = X^k - t \mathcal{A}^*y^k; \\
U^{k} &= X^{k-1} + t(\mathcal{A}^*y^{k-1} + S^k- C).
}
If the DRS \eqref{eq:DRS-sd} is first executed, we can obtain the following relationship for the ADMM  as
\bes{\label{eq:drs-admm2}
    t &= \mu;\\
    X^{k} &= \prox_{th}(Z^{k-1}); \\
    S^k &= \frac{1}{t}(X^k - Z^k); \\
    \mathcal{A}^*y^{k} &= \frac{1}{t}(X^{k} - X^{k-1}) - S^k +C.
}
The variable $y^k$ can be further computed from the last equation if the operator $\mathcal{A}$ is of full
row rank. Consequently, the strategies of the ADMM for updating $\mu$ can be
used in the DRS for modifying $t$ and vice versa. However, one should be careful
on computing the primal and dual infeasibilities of the DRS when the parameter $t$ is changed from $t_1$ to $t_2$ after one loop of
the DRS \eqref{eq:DRS-sd}. In this case, the next immediate
 update of the DRS should be %\CY{not clear what this is about}
\begin{equation}\label{eq:DRS-sd2}
\begin{split}
&X^{k+1} = \prox_{t_1h}\left(Z^k\right),\\
&U^{k+1} = \prox_{t_2f}\left(X^{k+1}-\frac{t_2}{t_1}(Z^k-X^{k+1})\right), \\
&Z^{k+1} = \frac{t_2}{t_1}(Z^k-X^{k+1} )+ U^{k+1}. \\
\end{split}
\end{equation} 
%After the above transformation between variables of the DRS and the ADMM, we
Thereafter, the original iterations \eqref{eq:DRS-sd} can still be used for the
fixed $t_2$.
%  \CY{is this about rescaling $t$ to accelerate convergence? need more explanation about the source of difficulty and how it is overcome.}

\subsection{Application to the 2-RDM}
 %To solve the dual SDP \eqref{eq:dualSDP}, we introduce the ADMM method, which is a generic method for SDP in \cite{Wen2010a}. In fact, it 
 %The ADMM has been used to
The ADMM has been successfully used to solve the 2-RDM problem in  \cite{Mazziotti2011} where the method is refereed to as the boundary point method. 
To apply ADMM to solve ~\eqref{eq:dualSDP}, we first write the augmented 
Lagrangian function as
\begin{equation}
\begin{split}
	L(y,S_j;X_j,x) &= b^Ty + \sum_{j=1}^l\iprod{X_j}{S_j - \cA_j^* y + C_j} + \iprod{x}{c-B^Ty}  \\ 
	&+\frac{\mu}{2}(\sum_{j=1}^l \|S_j - \cA_j^* y + C_j\|_F^2 +
    \|c-B^Ty\|_2^2),
\end{split}
\end{equation} 
 where $X_j$ and $x$ are Lagrangian multipliers and $\mu>0$ is a
 penalty parameter.
The the $k$th iteration of ADMM consists of the following sequence of steps:
\begin{equation} \label{eq:ADMM-2rdm}
	\begin{split}
		&y^{k+1} = {\arg\min}_y L(y,S_j^k;X_j^k,x^k), \\
		&S_1^{k+1}  = {\arg\min}_{0 \preceq S_1 \preceq I} L(y^{k+1},S_j;X_j^{k},x^{k}), \\
		&S_j^{k+1}  = {\arg\min}_{S_j \succeq 0} L(y^{k+1},S_j;X_j^{k},x^{k}), \ j = 2,\cdots,l, \\
		&X_j^{k+1}  = X_j^k + \mu (S_j^{k+1} - \cA_j^* y^{k+1} + C_j),  \ j = 1,\cdots,l, \\
		&x^{k+1}= x^k + \mu(c^{k+1} - B^Ty^{k+1}).
	\end{split}
\end{equation} 
%Define $ V_j = \cA_j^*y - C_j - X_j/\mu.  $
%Suppose that the eigen-decomposition of $V_j$ is $V_j = Q_j^*\Sigma_j Q_j $. Then we obtain
%\begin{equation}\label{eq:ADMM-2rdm}
%	\begin{split}
%		&y^{k+1} = (\sum_{j=1}^l \cA_j\cA_j^* + AA^T)^{-1}(\frac{1}{\mu}(\sum_{j=1}^l\cA_jX_j + Ax -b) + \sum_{j=1}^l\cA_j(S_j+C_j)+Ac), \\
%        &S_1^{k+1}  = Q_1^* \cP_{\{0 \leq x \leq 1\}}\Sigma_1 Q_1 , \\
%        &S_j^{k+1}  = Q_j^* \cP_{\{x \geq 0\}}\Sigma_j Q_j \ j = 2,\cdots,l, \\
%		&X_j^{k+1}  = X_j^k + \mu (S_j^{k+1} - \cA_j^* y^{k+1} + C_j)  \ j = 1,\cdots,l, \\
%		&x^{k+1}= x^k + \mu (c^{k+1} - A^Ty^{k+1}),
%	\end{split}
%\end{equation}
%where $\cP_{\Omega}$ is the projection on the set $\Omega$. 

The convergence of ADMM has been studied in 
\cite{gabay1976dual, eckstein1992douglas, boyd2011distributed, chen2017note,
Wen2010a}. 
%Note that the variables of \eqref{eq:dualSDP} can be put as two
%groups as $y$ and $\{S_j\mid
%j=1,\ldots,l\}$.  %$\{X_j\mid j=1,\ldots,l\}$
The following theorem, which establish the convergence of the ADMM 
to the solution of \eqref{eq:dualSDP} follows from the analysis 
given in \cite[Theorem 2.9]{Wen2010a}.
\begin{theorem} \label{thm:conv}
Suppose that the KKT points of \eqref{eq:dualSDP} exist.  Then the sequence of 
variables $(X_j^k,x^k, S_j^k,y^k)$ generated from the ADMM converge to a 
solution $(X_j^*,x^*, S_j^*,y^*)$ of \eqref{eq:dualSDP} from any starting 
point.  Furthermore, $\|\sum_{j=1}^l\iprod{X_j^k}{C_j}+c^Tx - b^Ty^k\|$, $\|\sum_{j=1}^l\cA_jX_j + Bx -b\|$, $\|\cA_j^*y^k-C-S_j^k\|_F$ and $\|B^Ty^k-c\|$ all converge to 0.
\end{theorem}

%=================================The Semi-smooth Newton method====================================

\section{The Semi-smooth Newton method} \label{sec:nt}
%In this section, we present a general algorithm to solve the SDP problems.  Firstly, we introduce the relationship between the ADMM and the DRS and extend a few efficient strategies from the ADMM to the DRS \cite{DR1956,LM1979,EB1992}. Then, we develop a semi-smooth  Newton method to overcome the slow convergence  of the ADMM.
%In this section,  we develop a semi-smooth  Newton method to overcome the slow convergence  of the ADMM.
%\subsection{Semismooth Newton Method}
Although Theorem~\ref{thm:conv} asserts that the ADMM (and consequently the 
DRS method due to its equivalence to the ADMM) converges from 
any starting point, the convergence can be slow, especially towards a
highly accurate approximation to the solution of \eqref{eq:dualSDP}.
In practice, we often observe a rapid reduction in the objective 
function, infeasibility and duality gap in the first few iterations. 
However, the reduction levels off after the first tens or hundreds 
of iterations.  To accelerate convergence and obtain a more 
accurate approximation, we consider a second-order method.

The DRS can be characterized as a
fixed-point iteration \eqref{eq:fix-DRS} for solving a system of nonlinear 
equations
\be \label{eq:fix1}
F(Z) = \prox_{th}(Z) - \prox_{tf}(2\prox_{th}(Z) - Z) = 0, 
\ee
where $Z \in \R^{n \times n}$. Moreover, the solution of \eqref{eq:fix1} is also
an optimal solution to
\eqref{eq:primSDP-s} and vice versa. Hence, we will focus on more efficient ways to
solve the equations \eqref{eq:fix1}. 
%\CY{Is the solution to this equation the same as the solution to SDP? i.e. is
%this equation a first order necessary condition of SDP?} Yes, it is equivalent
%For the simplicity of solving the subproblems in the DRS, we make an 
%assumption that $\cA\cA^*  =I$. 
To simplify the derivation of the new method to be presented,
we first make the following assumption.
\begin{assumption} The operator $\cA$ in \eqref{eq:primSDP-s} satisfies
  $\cA\cA^*  =I$ and the Slater condition holds. That is, there exits
  $X\succ 0$ such that $\cA X=b$.
  \end{assumption}

The first part of the assumption implies that $\cA$ has full row rank. 
It is satisfied in many SDPs including \eqref{eq:dualSDP} after a suitable
transformation of $\cA$.
 %rearrangement of the constraints.

\subsection{Generalized Jacobian} 
Before we discuss how to solve \eqref{eq:fix1}, let us first examine the 
structure of the generalized Jacobian of $F(Z)$.
Using the definition of $f(x)$ and $h(x)$ given in \eqref{eq:def-fh},
we can write down the explicit forms of $\prox_{tf}(Y)$ and $\prox_{th}(Z)$ as 
\begin{eqnarray*}
\prox_{tf}(Y) &=& (Y + t C) - \cA^* (\cA Y+ t\cA C-b), \\
\prox_{th}(Z) &=& Q_\dagger \Sigma_+ Q_\dagger^T, 
\end{eqnarray*}
and
where
\[ Q \Sigma Q^T = \begin{pmatrix} Q_\dagger & Q_\ddagger
\end{pmatrix}\begin{pmatrix} \Sigma_+ & 0 \\ 0 & \Sigma_- \end{pmatrix}
    \begin{pmatrix} Q_\dagger^T \\ Q_\ddagger^T \end{pmatrix} \] is
        the spectral decomposition of the matrix $Z$, and the diagonal 
matrices $\Sigma_+$ and $\Sigma_-$ contain the nonnegative and negative eigenvalues of
$Z$.

Although $F$ is not differentiable, its generalized subdifferential still
exists. Since $F$ is locally Lipschitz continuous,  it  can be verified %follows from Rademacher's theorem
that
$F$ is almost differentiable everywhere. We next introduce the concepts of generalized subdifferential.
\begin{definition}\label{def:GD}
Let $F$ be locally Lipschitz continuous at $X\in\cO$, where $\cO$ is an open set. Let $D_F$ be the set of differentiable
points of $F$ in $\cO$. The \emph{B-subdifferential} of $F$ at $X$ is defined by
\[
\partial_BF(X):=\left\{\lim\limits_{k\rightarrow\infty}F'(X^k)| X^k\in D_F, X^k\rightarrow X\right\}.
\]
The set 
$\partial F(x)=\emph{co}(\partial_BF(x))$
is called Clarke's \emph{generalized Jacobian}, where $\emph{co}$ denotes the convex hull.
\end{definition}

% B-subdifferential is a common definition and it is defined in the above
% definition

It can be shown that the generalized Jacobian matrix associated with 
the second term of $F(Z)$ in \eqref{eq:fix1} %\CY{or just $\prox_{tf}$?} 
has the form
\be 
\label{eq:def-D}
\cD \equiv \partial \prox_{tf}((2\prox_{th}(Z)-Z))
=  \cI- \cA^*\cA,
\ee
where $\cI$ is the identity operator.
Similar to the convention used in \cite{Zhao2009}, we define a generalized 
Jacobian operator $\cM(Z)\in \partial \prox_{th}(Z)$ %\CY{what is $\cP_K(Z)$?} 
in terms of its application to an $n$-by-$n$ matrix $S$ that yields
\be \label{eq:def-M}
\cM(Z)[S] = Q(\Omega \circ (Q^T S Q))Q^T,  \forall S \succeq 0,
\ee
where $Q\Sigma Q^T$ is the eigen-decomposition of $Z$ with
$\Sigma = \diag(\lambda_1,\cdots,\lambda_n)$, 
%and $\lambda_i$'s, $i = 1, \cdots, n$ are its eigenvalue. 
and 
\[
\Omega = 
\left[
\begin{matrix}
E_{\alpha\alpha} & k_{\alpha \bar{\alpha}} \\
k_{\alpha \bar{\alpha}}^T & 0  
\end{matrix}
\right],
\]
with $\alpha = \{i | \lambda_i > 0\}$, 
$\bar{\alpha} = \{1,\ldots,n\} \setminus \alpha$ and $E_{\alpha\alpha}$ being
a matrix of ones and $k_{ij} = \frac{\lambda_i}{\lambda_i - \lambda_j}, i \in \alpha, j \in \bar{\alpha}.$ The $\circ$ symbol appeared in \eqref{eq:def-M} 
denotes a Hadamard product. It follows from \eqref{eq:fix1}, \eqref{eq:def-D} and 
\eqref{eq:def-M} that the generalized Jacobian of $F(z)$ can be written as 
\be\label{eq:jacob}
\cJ(Z)= \cM(Z) + \cD(I-2\cM(Z)).
\ee

The function $F$ given in \eqref{eq:fix1} is strongly semi-smooth \cite{Mifflin1977, QS1993} and monotone, which is important for establishing
the positive semidefinite nature of its $B$-subdifferential. The precise 
definitions of these properties are given below.
\begin{definition}\label{def:semi}
Let $F$ be a locally Lipschitz continuous function in a domain $\cO$. We say
that $F$ is  semi-smooth at $x\in\cO$ if (i) $F$ is directionally differentiable at
$x$; (ii) for any $z\in\cO$ and $\cJ\in\partial F(x+z)$,
\be \label{eq:smsm}
\|F(x+z)-F(x)-\cJ[z]\|_2=o(\|z\|_2)\quad\mbox{as}\ z\rightarrow 0.
\ee
The function $F$ is said to be strongly  semi-smooth if $o(\|z\|_2)$ in \eqref{eq:smsm} is replaced by $O(\|z\|_2^2)$. It is called
\textit{monotone} if $ \iprod{x-y}{F(x)-F(y)}\geq 0,\ \mbox{for all}\
x,y\in\R^n$.
\end{definition}

The next lemma characterizes the fixed point map given in \eqref{eq:fix1} 
and its generalized Jacobian matrix.
\begin{lemma}\label{lemma:semimono}
The function $F$ in \eqref{eq:fix1} is strongly semi-smooth and monotone. Each element of B-subdifferential $\partial_B F(x)$ of $F$ is positive semidefinite.
\end{lemma}
\begin{proof}
The strongly semi-smoothness of $F$ follows from the derivation given 
in~\cite{RW1998, SS2002} to establish the semi-smoothness of proximal mappings. In fact, the projection over a polyhedral set is
strongly semi-smooth \cite[Example 12.31]{RW1998}  and the projections over
symmetric cones are proved to be strongly semi-smooth in \cite{SS2002}. Hence,
$\prox_{tf}(\cdot)$ and $\prox_{th}(\cdot)$ are strongly semi-smooth. Since
strongly semi-smoothness is closed under scalar multiplication, summation and
composition, the function $F$ is strongly semi-smooth.

 It has been shown in \cite{LM1979} that the operator $T_{\textrm{DRS}}$ is
firmly nonexpansive. Therefore, $F$ is firmly nonexpansive, hence monotone
\cite[Proposition 4.2]{BC2011}.  The positive semidefiniteness simply follows
from Lemma 3.5 in~\cite{xiao2016regularized}.
\end{proof}
\subsection{Computing the Newton direction}
Using the expression given in \eqref{eq:jacob}, we can now discuss
how to compute the Newton direction efficiently.  At a given iterate $Z^k$, 
we compute a Newton direction $S^k$ by solving the equation
\be\label{eq:r-newton}
(\cJ_k+\mu_k\cI)[S^k]=-F^k,
\ee
where $\cJ_k\in\partial_BF(Z^k)$, $F^k=F(Z^k)$, $\mu_k=\lambda_k\|F^k\|_2$ and 
$\lambda_k>0$ is a regularization parameter.
The equation \eqref{eq:r-newton} is well-defined since each element of B-subdifferential $\partial_B F(x)$ of $F$ is positive
semidefinite and the regularization term $\mu_k I$ is chosen such that
$\cJ_k+\mu_k \cI$ is invertible. From a computational view, it is not practical to solve the linear system (\ref{eq:r-newton}) exactly. Therefore, we
seek an approximate step $S^k$ by solving (\ref{eq:r-newton}) approximately 
so that
\be\label{eq:residual}
\|r^k\|_F\leq\tau\min\{1,\lambda_k\|F^k\|_F \|S^k\|_F\},
\ee
where 
\be\label{eq:rk}
r^k:=(\cJ_k+\mu_k \cI)[S^k]+F^k
\ee
is the residual and $0<\tau<1$ is some positive constant. 

Since the $\cJ_k$ matrix in \eqref{eq:r-newton} is nonsymmetric, and
its dimension is large, we apply the binomial inverse theorem to transform 
\eqref{eq:r-newton} into a smaller symmetric
system. If we vectorize the matrix $S$, %\CY{do we mean $d_s$? $d$ is not used below, it is used to denote the number of orbitals earlier} with respect to the standard bases of $\R^{n \times n}$, 
 then the operators $\cM(Z)$ and $\cD$ 
can be expressed as matrices
\[
M(Z) = \tilde{Q} \Lambda \tilde{Q}^T
\
\text{and} 
\
D = I - A^TA，
\]
respectively, where $\tilde{Q} = Q\otimes Q$, $\Lambda = \diag(\ivec(\Omega))$, $I$ is the identity matrix and $A$ is the matrix form of $\cA$. % under the standard bases. 
Let $W=I-2M(Z) = \tilde{Q}(I - 2\Lambda)\tilde{Q}^T$ and
$H=\tilde{Q}((\mu_k+1)I - \Lambda)\tilde{Q}^T$. Then the matrix form of $\cJ_k+\mu_k I $ can be written as $J_k + \mu_k I = H - A^TAW$. It follows from the 
binomial inverse theorem that
\beaa
(J_k+\mu_k I)^{-1} &=& (H - A^T A W)^{-1} \\
&=&  H^{-1} +  H^{-1} A^T (I- AW  H^{-1}A^T)^{-1} A WH^{-1}.
\eeaa
%Both $H^{-1}$ and $WH^{-1}$ has a sandwiches structure, i.e., $H^{-1} = \tilde{Q}L_1\tilde{Q}^T$ and $WH^{-1} = \tilde{Q}L_2\tilde{Q}^T$, where $L_1$ and $L_2$ are diagonal matrices. 
%with diagonal entries
%\[
%(L_1)_{ii}=\left\{
%\begin{array}{ll}
%\frac{1}{\mu} ,\quad & (\Lambda)_{ii}=1,\\
%\frac{1}{\mu+1-\omega} ,\quad & (\Lambda)_{ii} = \omega,\\
%\frac{1}{1+\mu},\quad & (\Lambda)_{ii} = 0.\\
%\end{array}
%\right.
%\text{and} \quad
%(L_2)_{ii}=\left\{
%\begin{array}{ll}
%-\frac{1}{\mu} ,\quad & (\Lambda)_{ii}=1,\\
%\frac{1-2\omega}{\mu+1-\omega} ,\quad & (\Lambda)_{ii} = \omega,\\
%\frac{1}{1+\mu},\quad & (\Lambda)_{ii} = 0.\\
%\end{array}
%\right.
%\]
Define \be \label{eq:def-T} T = \tilde{Q}L\tilde{Q}^T, \ee
where $L$ is a diagonal matrix with diagonal entries $L_{ii}= \frac{\Lambda_{ii}
\mu_k}{\mu_k+1 -\Lambda_{ii}}$.
%\[
%L_{ii}(z)=\left\{
%\begin{array}{ll}
%  1 ,\quad & \mbox{ if } (\Lambda)_{ii}=1,\\
%\frac{\omega \mu_k}{\mu_k+1 -\omega},\quad &  \mbox{ if } (\Lambda)_{ii} = \omega \in (0,1),\\
%0,\quad &  \mbox{ if }  (\Lambda)_{ii} = 0.\\
%\end{array}
%\right.
%\]
By using the identities $H^{-1} = \frac{1}{\mu_k+1}I + \frac{1}{\mu_k(\mu_k+1)}T $ and $WH^{-1} =
\frac{1}{1+\mu_k}I - (\frac{1}{\mu_k}+\frac{1}{\mu_k+1})T $, we can
further obtain 
\bea\label{eq:nleqn1}
&& (J_k+\mu_k I)^{-1}  \\ \nonumber
&=&  \frac{\mu_k I + T}{\mu_k(\mu_k+1)} \left(I+ A^T
 \left(\frac{\mu_k^2}{2\mu_k+1}I+A T A^T\right)^{-1} A (\frac{\mu_k}{2\mu_k+1}I -
T )\right).
\eea
As a result, the solution of \eqref{eq:r-newton} can be obtained by first solving
the following symmetric linear equation
 %\CY{may need to give the dimension of $A$ to give a sense how much we save}
\be\label{eq:nleqn2}
\left(\frac{\mu_k^2}{2\mu_k+1}I+ATA^T \right)d_s =  a,
\ee
where $a = -A (\frac{\mu_k}{2\mu_k+1}I - T)\svec(F^k)$, by
an iterative method such as the conjugate gradient (CG) method or 
the symmetric QMR method. Note that the size of the coefficient matrix of \eqref{eq:nleqn2} is $m \times m$
while that of \eqref{eq:r-newton} is $n^2 \times n^2$, where $m$ usually is
 much smaller than $n^2$. Then we use the following expression
to recover
\[
S^k =  \frac{1}{\mu_k(\mu_k+1)}(\mu_k \cI + \cT) [-F^k+ \cA^* d_s],
\] 
where $\cT$ is the operator form of $T$ in \eqref{eq:def-T}.
Specifically, applying $\cT $ to a matrix $S$ yields
\[
\cT(Z)[S] = Q(\Omega_0 \circ (Q^T S Q))Q^T,  \forall S \succeq 0,
\] 
where 
$ \Omega_0 = 
\left[
\begin{matrix}
E_{\alpha\alpha} & l_{\alpha \bar{\alpha}} \\
l_{\alpha \bar{\alpha}}^T & 0  
\end{matrix}
\right]$, 
and $l_{ij} = \frac{\mu_k k_{ij}}{\mu_k + 1 - k_{ij}}  $. 

Let $\Upsilon = \cT(Z)[S]$. We can then use 
the same techniques used in \cite{Zhao2009} to express $\Upsilon$ as
multiplication:
\be\label{eq:mv1}
\Upsilon = [Q_{\alpha} Q_{\bar{\alpha}}]
\left[\begin{matrix}
	Q_{\alpha}^TSQ_{\alpha} & l_{\alpha \bar{\alpha}} \circ Q_{\alpha}^TSQ_{\bar{\alpha}} \\
	l_{\alpha \bar{\alpha}}^T \circ Q_{\bar{\alpha}}^TSQ_{\alpha} & 0
\end{matrix}\right]	
\left[\begin{matrix}
Q_{\alpha}^T \\
Q_{\bar{\alpha}}^T
\end{matrix}\right]	= G + G^T,	
\ee
where $G = Q_{\alpha}(\frac{1}{2}(UQ_{\alpha}^T)+ l_{\alpha \bar{\alpha}} \circ
(UQ_{\bar{\alpha}}) )$ with $U = Q_{\alpha}^TS$. {The number of floating point 
operations (flops) required to compute $\Upsilon$ is $8|\alpha|n^2$. 
If $|\alpha|$ is large, we can compute $\Upsilon$ via the equivalent
expression $\Upsilon = S-Q((E-\Omega_0) \circ (Q^T S Q))Q^T$, which requires
$8|\bar{\alpha}|n^2$ flops.} 

Therefore, using the expression \eqref{eq:mv1} allow us to obtain an 
approximate solution to \eqref{eq:r-newton} efficiently whenever 
$|\alpha|$ or $|\bar{\alpha}|$ is small.
We summarize the procedure for solving the Newton equation \eqref{eq:r-newton} 
approximately in Algorithm~\ref{algo:linear}. 

%\CY{how does the complexity compare with standard algorithms, make a connection between the size of $\alpha$ and the low rank property established in Theorem 2.1?}.

\begin{algorithm2e}[!htb]\label{algo:linear}
\caption{Solving the linear system \eqref{eq:r-newton}}
Compute $a = -\cA (\frac{\mu_k}{2\mu_k+1}\cI - \cT)F^k$ \;
Use the CG or symmetric QMR method to solve
$(\frac{\mu_k^2}{2\mu_k+1}\cI+\cA\cT\cA^*)d_s =  a$ inexactly, where the matrix-vector multiplication is computed by \eqref{eq:mv1} \;
Compute the Newton direction $S^k =  \frac{1}{\mu_k(\mu_k+1)}(\mu_k \cI + \cT) (-F^k+ \cA^* d_s)$.
\end{algorithm2e}

%%%%%%%%%%%%%%%%%%%%%%%%%%%%%%%%%%%%%%%%
%%%%%%%%%%%%%%%%%%%%%%%%%%%%%%%%%%%%%%%%
\subsection{Strategy for updating $Z^k$} A few safeguard strategies are developed in order to stabilize the semi-smooth
Newton algorithm and maintain global convergence. Let $U^{k} = Z^{k} + S^k$ be a
new trial point from the Newton step and set $\xi_0 = \|F(Z^0)\|_2$.  If the residual
 $\|F(U^k)\|$ is
sufficiently decreased, i.e., $\|F(U^{k})\|_2 \le \nu \xi_{k}$ 
with $0<\nu<1$, then we update  
\be \label{eq:Newton-succ} Z^{k+1} = U^k, \; \xi_{k+1}=\|F(U^k)\|_2 \mbox{ and } \lambda_{k+1}=
\lambda_k. 
\ee
%\lambda_{k+1}= \lambda_k. \quad \mbox{[Newton
%step]}\ee

%\comm{Please keep my writing on the algorithmic description. We can not figure out a simple form with convergence guarantee.}

Otherwise, we examine the
ratio
\be\label{eq:rho}
\rho_k=\frac{-\iprod{F(U^k)}{S^k}}{\|S^k\|_F^2}
\ee
to decide how to update $Z^k$, $\xi_k$ and $\lambda_k$.  
 If $\rho_k\geq\eta_1$ for some $\eta_1$ that satisfies 
$0<\eta_1\le\eta_2<1$, we compute a new trial point using so-called \textit{hyperplane projection} 
step \cite{SS1999}:
\be\label{eq:HP-step}
V^{k}=Z^k-\frac{\iprod{F(U^k)}{Z^k-U^k}}{\|F(U^k)\|_2^2}F(U^k).
\ee
 Assume that the set of the optimal solutions
of  \eqref{eq:fix1} is $\Omega$. 
  By the monotonicity of $F$, for any optimal solution $Z^*$, one always has
 $ \iprod{F(U^k)}{Z^*-U^k}\leq 0$.  If the ratio $\rho_k > 0$, then we have 
$ \iprod{F(U^k)}{-S^k}>0$.  
Therefore, the hyperplane
\[
H_k:=\{Z\in\R^{n \times n}|\iprod{F(U^k)}{Z-U^k}=0\}
\]
strictly separates $Z^k$ from the solution set $\Omega$. 
 %It follows from this observations that~\cite{SS1999} 
It is easy to show that the point $V^{k}$ defined in \eqref{eq:HP-step} is the projection of $Z^k$ onto the
hyperplane $H_k$ and it is closer to $Z^*$ than $Z^k$.  This projection step can be used to correct a 
potentially poor Newton step. Hence, we set $Z^{k+1}=V^k$ if $\|F(V^k)\|_2 \leq
\|F(Z^k)\|_2$. Otherwise, we still take a DRS iteration, i.e., $W^k =
Z^k-F(Z^k)$.
In summary, we set
\be\label{eq:zk}
Z^{k+1}=\begin{cases}
V^k,\ &\textrm{if} \; \rho_k\geq\eta_1 \; \mbox{and} \; \|F(V^k)\|_2 \leq
\|F(Z^k)\|_2, \; \mbox{ [projection step]}\\
W^k, &\textrm{if}\; \rho_k\geq\eta_1 \; \mbox{and} \; \|F(V^k)\|_2 >
\|F(Z^k)\|_2, \; \mbox{ [DRS step]}\\
Z^k, &\textrm{if}\; \rho_k<\eta_1, \qquad \qquad\qquad \qquad\qquad \quad\quad  \mbox{ [unsuccessful step]}
\end{cases}
\ee
Then the  parameters $\xi_{k+1}$ and 
$\lambda_{k+1}$ are updated as  
\begin{equation}\label{eq:lambda}
\xi_{k+1}=\xi_k, \quad % \lambda_{k+1}= \lambda_k \mbox{ if } \|F(u^k)\|_2 \le \nu \|F(\bar{u})\|_2,
  \lambda_{k+1} \in
\begin{cases}
%\{\lambda_k\},\quad &\textrm{if}\ \|F(u^k)\|_2 \le \nu \|F(\bar{u})\|_2,\\[8pt]
(\uline{\lambda},\lambda_k), &\textrm{if}\ \rho_k\geq\eta_2,\\[8pt]
[\lambda_k,\gamma_1\lambda_k], &\textrm{if}\ \eta_1\leq\rho_k<\eta_2,\\[8pt]
(\gamma_1\lambda_k,\gamma_2\lambda_k], &\textrm{otherwise,}
\end{cases}
\end{equation}
where  $1<\gamma_1 <\gamma_2$ and $\uline{\lambda}>0$ is a small positive constant.
%These parameters determine how aggressively the regularization parameter is increased when an iteration is unsuccessful or it is decreased when $\rho_k\geq\eta_2$. The complete approach to solve (\ref{eq:monoNE}) is summarized in Algorithm \ref{algo:newton}.

%\CY{this is not an updating formula}  This is a common notation 

%The step \eqref{eq:HP-step} is in fact the hyperplane projection technique developed in \cite{SS1999} for
%solving the system \eqref{eq:fix1} in order to guarantee the global convergence
%of the semi-smooth Newton method.  Assume that the set of the optimal solutions
%of  \eqref{eq:fix1} is $\Omega$. 
%  By the monotonicity of $F$, for any optimal solution $Z^*$, one always has
% $ \iprod{F(U^k)}{Z^*-U^k}\leq 0$.  If the ratio $\rho_k > 0$, then we have 
%$ \iprod{F(U^k)}{-S^k}>0$.  
%Therefore, the hyperplane
%\[
%H_k:=\{Z\in\R^{n \times n}|\iprod{F(U^k)}{Z-U^k}=0\}
%\]
%strictly separates $Z^k$ from the solution set $\Omega$. 
% %It follows from this observations that~\cite{SS1999} 
%It is easy to show that the point $V^{k}$ defined in \eqref{eq:HP-step} is the projection of $Z^k$ onto the
%hyperplane $H_k$.  This projection step can be used to correct a 
%potentially poor Newton step so that global convergence can be achieved.

 The complete approach to solve (\ref{eq:fix1}) is summarized
in Algorithm \ref{algo:newton}.

\begin{algorithm2e}[htb!]
\caption{\emph{An Adaptive Semi-smooth Newton (ASSN) method for SDP}}
\label{algo:newton}
   Give  $0<\tau,\nu<1$,  $0<\eta_1\leq\eta_2<1$ and
  $1<\gamma_1\leq\gamma_2$ \; Choose $Z^0$ and $\varepsilon>0$. Set $k=0$ and
  $\xi_0=\|F(Z^0)\|_2$\;
  \While{not ``converged''}{
%If $d^k=0$, stop\;
Select $J_k\in\partial_B F(Z^k)$\;
Solve the linear system (\ref{eq:r-newton})
approximately such that $S^k$ satisfies (\ref{eq:residual})\;
 Compute $U^k=Z^k+S^k$ and calculate the ratio $\rho_k$ as in (\ref{eq:rho})\;
 If $\|F(U^k)\|_2 \le \nu
 \xi_k$, Set update $Z^{k+1}$, $\xi_{k+1}$ and $\lambda_{k+1}$ according to
\eqref{eq:Newton-succ}. Otherwise, set them  according to (\ref{eq:zk}) and
(\ref{eq:lambda}), respectively\; 
  Set $k=k+1$\;
  }
\end{algorithm2e}

%The global convergence of Algorithm \ref{algo:newton} follows from  \cite{xiao2016regularized}.

The following theorem establishes the global convergence of Algorithm~\ref{algo:newton}.

%In \cite{xiao2016regularized} we show the following global convergent theorem for a general method .
%\begin{theorem}\label{th:newton}
% Assume that $F:\R^{s} \rightarrow\R^{s}$ is strongly semi-smooth and monotone.
%Suppose that there exists a constant $c_1>0$ such that 
%$\|J_k\|\leq c_1$ for any $k\geq 0$ and any $J_k\in\partial_{B}F(Z^k)$. Then either the semi-smooth Newton method in  \cite{xiao2016regularized}  terminates in finitely many iterations or $\{Z^k\}$ converges to some point $\bar{Z}$ such that $F(\bar{Z})=0$.
%\end{theorem}

\begin{theorem}\label{th:newton-sdp}
Suppose that $\{Z_k\}$ is a sequence generated by Algorithm
\ref{algo:newton}. Then the residuals of $\{Z_k\}$ converge to 0, i.e., $\lim_{k \rightarrow \infty} ||F(Z_k)|| = 0$.
\end{theorem}
\begin{proof}
 The strongly semi-smoothness and monotonicity has been shown in Lemma
 \ref{lemma:semimono} and the firmly non-expansiveness of fixed-point mapping $T_{\mathrm{DRS}} = I - F$ has been shown in \cite{LM1979}.
 The proof is completed according to Theorem 3.10 in \cite{xiao2016regularized}.
%According to Theorem 3.10 in \cite{xiao2016regularized},
% we only need to check the boundness of Jacobian \eqref{eq:jacob}.  Recall that  the
% matrix form of Jacobian \eqref{eq:jacob} is 
%\be
%J(Z) = M(Z) + D(I-2M(Z)).
%\ee
%Since all elements of $\Omega$ are between 0 and 1, $\|M(Z)\|$
%and $\|(I-2M(Z))\|$ are both less than 1. Let $c = 1+\| I - A^TA\|$.
%Due to the fact that $\|D\|$ is a constant, we have 
%\be
%\|J(Z)\| \leq \|M(Z)\| + \|D\|\|I-2M(Z)\| \leq c,
%\ee
%which completes the proof.
\end{proof}

%=================================Numerical Results====================================

\section{Numerical Results} \label{sec:num}
In this section, we demonstrate the effectiveness of the semi-smooth Newton
method presented in the previous section. We implemented 
the algorithm mostly in MATLAB. Our codes are built based on
+SDPNAL \cite{Zhao2009}, SDPNAL+ \cite{Chaykin2016} and ADMM+ \cite{MR3342702}, and use most of the key implementation details and subroutines in these solvers. Some parts of the code are written in 
the C Language and interfaced with MATLAB through MEX-files.  
All experiments are performed on a
 single node of a PC cluster, where each node has two Intel Xeon 2.40GHz CPUs
 with 12 cores and 256GB RAM. 

The test dataset is provided by Professor Maho Nakata and Professor Mituhiro
Fukuta. %These problems have been tested in \cite{Nakata2008}. 
The detailed information about the dataset such as the basis sets used to 
discretize molecular orbitals, the geometries of the molecules etc.
can be found in \cite{Nakata2008}. Since the original dataset only
takes into account the spin symmetry, it does not specify 
additional block diagonal structures introduced by spatial symmetry 
of the molecular orbitals within each spin matrix block of the  
variables.
We preprocess the dataset to identify these diagonal blocks automatically
through matrix reordering. % (using approximate minimum degree ordering~\cite{}).  
Our solver takes advantage of these block diagonal structures to 
reduce the complexity of the computation as described in
subsection \ref{sec:sdp-2rdm}. We applied the semi-smooth Newton
algorithm to the SDP formulation of the 2-RDM minimization problem with 
four different groups of N-representability conditions labeled as 
PQG, PQGT1, PQGT1T2, PQGT1T2'. The letters and numbers in each label simply 
indicate the N-representability conditions included in the SDP constraints. 
For example, PQGT1T2' means that the P, Q, G, T1, T2' conditions 
are included.

We compare the semi-smooth Newton's method proposed in this 
paper with the state-of-the-art Newton-CG augmented Lagrangian method 
implemented in the SDPNAL software package~\cite{Zhao2009}. 
We choose to compare with SDPNAL instead of SDPNAL+ \cite{Chaykin2016} because
our test examples are in standard SDP forms for which SDPNAL works better
than SDPNAL+ in our numerical experiments. The interior point methods are not
included in the comparison because they usually performs worse than SDPNAL. The stopping rules and a number of
parameters used in SDPNAL are set to their default values. 
We measure accuracy by examining four criteria: the primal infeasibility $\eta_p$
and the dual infeasibility $\eta_q$ that are defined by \eqref{eq:pdinf}, the gap $\eta_g$ between the primal and dual objective functions  
\be
\eta_g = \frac{|b^Ty - \tr (C^T X)|}{\max(1, \tr(C^T X))},
\ee
and the difference between the 2-RDM energy and full CI energy defined by
\be
\mathrm{err} = b^Ty - \mathrm{energy_{full CI}},
\label{eq:fcierr}
\ee
where $\mathrm{energy_{full CI}}$ values are taken from
\cite{Nakata2008}. 
The last criterion is often used in quantum chemistry to assess the 
accuracy of an approximation model. It is used here to also 
assess the effectiveness in including additional N-representability
conditions in the 2-RDM formulation of the ground state energy minimization 
problem. In the following tables, we use a short notation for the exponential
form. For example, -4.8-3 means $-4.8\times 10^{-3}$.

We experimented with two versions of the semi-smooth Newton methods. 
The difference between these two versions is in the stopping rules and 
how the parameter $\mu$ is updated. The first version, which is called
SSNSDPL, uses a stopping rule that is similar to the one used in SDPNAL. 
Specifically, in this version, the iterative procedure is terminated when 
$\eta_p < 3\times 10^{-6}$ and $\eta_d < 3\times 10^{-7}$ so that it can
achieve higher accuracy than that produced by SDPNAL. 
 %\CY{why? what's used in SDPNAL?}: wzw: it is explained in the above paragraph 
The choice of these parameters makes SSNSDPL comparable to SDPNAL. 
Another version, which is called SSNSDPH, uses a more stringent
stopping rule that requires $\eta_p < 1\times 10^{-4}$ and $\eta_d < 1\times
10^{-9}$.  In this version, the primal
infeasibility $\eta_p$ is allowed to be larger so that the algorithm converges
more rapidly. % without affecting the value of ``$\mathrm{err}$".
%\CY{not sure what that means} 
The dual variables are required to be more accurate since we ultimately retrieve the desired 1-RDM and 2-RDM from the dual variables.  This version can reach a ``$\mathrm{err}$" level that is close to the one reported in~\cite{Nakata2008}. In this version, we also make the penalty parameter $\mu$
larger so that the stopping rules can be easier satisfied.

In Table \ref{tab:sp}, we compare the performance of SSNSDPL when it
is applied to the orginal dataset provided in~\cite{Nakata2008} and 
our preprocessed data that identifies additional block diagonal
structures through permutation. We can
see that the CPU time can be reduced by at least a factor of three on most 
examples labeled with PQGT1T2 and PQGT1T2'. For $\mathrm{C}$ atom and 
$\mathrm{F^{-}}$ system that exhibit high spatial symmetry, 
the CPU time measured in seconds (the column labeled by t in 
Table~\ref{tab:sp}) can be reduced by a factor of roughly six for SDP's 
that include the PQGT1T2 and PQGT1T2' conditions.  These experiments 
illustrate the importance of exploiting spatial symmetry to identify 
block diagonal structures in the approximate solution and consequently 
reduce the computational cost significantly. For the problems that only 
include the PQG and PQGT1 conditions, the amount of improvement is less 
spectacular, because the sizes of the diagonal blocks in these examples are 
small. %The cost in these small examples does not mainly focus on large scalability. 
In fact, the larger the blocks in Table \ref{tab:blk} is, the more significant
 effectiveness of the symmetry is. Thereafter, all experiments are performed on
 the preprocessed data. 
 
 %\CY{need some rewording here. Need to be clear about what we mean by blocks}
%
%\begin{scriptsize}
\begin{table}
{
\setlength{\tabcolsep}{1pt}
\footnotesize
\begin{longtable}{|c|c| c|c|c|c|c|c| c|c|c|c|c|c|}
\caption{The comparison of the performance on the original and proprocessed
SDPs. The number -4.8-3 means $-4.8\times 10^{-3}$}\label{tab:sp}\\
\hline 
\multicolumn{1}{|c|}{}& \multicolumn{1}{|c|}{}
 & \multicolumn{6}{|c|}{preprocessed SDP}
 & \multicolumn{6}{|c|}{original SDP}\\ 
\hline 
system & condition
 & err & $\eta_p$ & $\eta_d$ & $\eta_g$ & it & t 
 & err & $\eta_p$ & $\eta_d$ & $\eta_g$ & it & t \\ \hline 
\endhead 
\hline 
\endfoot 
$\mathrm{C}$& PQG
 &-4.8-3 & 1.3-6 & 3.0-7 &  1.0-5 &   335 & 136
 &-4.7-3 & 3.2-7 & 2.7-7 &  9.8-6 &   163 & 82\\  
$\mathrm{C}$& PQGT1
 &-3.8-3 & 5.9-7 & 2.9-7 &  7.4-6 &   217 & 129
 &-3.8-3 & 5.3-7 & 2.9-7 &  7.0-6 &   157 & 189\\  
$\mathrm{C}$& PQGT1T2
 &-9.9-4 & 9.3-7 & 2.8-7 &  7.7-6 &   229 & 366
 &-9.5-4 & 1.1-6 & 3.0-7 &  6.5-6 &   219 & 1915\\  
$\mathrm{C}$& PQGT1T2'
 &-5.4-4 & 1.2-6 & 3.0-7 &  7.7-6 &   226 & 361
 &-5.4-4 & 1.7-6 & 3.0-7 &  1.2-5 &   235 & 2051\\  
$\mathrm{CH}$& PQG
 &-1.3-2 & 1.3-6 & 3.0-7 &  9.5-6 &   126 & 77
 &-1.3-2 & 2.1-6 & 2.0-7 &  8.2-6 &   196 & 167\\  
$\mathrm{CH}$& PQGT1
 &-1.0-2 & 1.7-7 & 2.7-7 &  9.7-6 &   122 & 220
 &-1.0-2 & 9.7-7 & 2.7-7 &  8.5-6 &   118 & 446\\  
$\mathrm{CH}$& PQGT1T2
 &-2.5-3 & 1.9-6 & 3.0-7 &  1.0-5 &   271 & 2008
 &-2.5-3 & 4.0-7 & 3.0-7 &  1.1-5 &   268 & 6351\\  
$\mathrm{CH}$& PQGT1T2'
 &-1.1-3 & 5.1-7 & 2.9-7 &  9.9-6 &   294 & 2041
 &-1.1-3 & 3.9-7 & 2.9-7 &  9.6-6 &   282 & 6350\\  
$\mathrm{F^-}$& PQG
 &-1.3-2 & 6.3-7 & 1.8-7 &  8.5-6 &   119 & 72
 &-1.2-2 & 1.6-6 & 1.1-7 &  1.3-5 &   113 & 106\\  
$\mathrm{F^-}$& PQGT1
 &-9.2-3 & 1.9-6 & 2.6-7 &  1.3-5 &   144 & 211
 &-8.7-3 & 1.1-6 & 2.6-7 &  5.4-6 &   149 & 718\\  
$\mathrm{F^-}$& PQGT1T2
 &-2.6-3 & 2.1-6 & 2.7-7 &  2.2-5 &   225 & 1325
 &-2.7-3 & 1.3-6 & 2.8-7 &  1.7-5 &   235 & 10735\\  
$\mathrm{F^-}$& PQGT1T2'
 &-2.0-3 & 1.2-6 & 2.7-7 &  1.1-5 &   249 & 1359
 &-1.9-3 & 1.2-6 & 3.0-7 &  1.2-5 &   234 & 9485\\  
$\mathrm{H_2O}$& PQG
 &-1.9-2 & 2.2-6 & 1.2-7 &  1.6-6 &    87 & 93
 &-1.9-2 & 1.3-6 & 1.0-7 &  2.3-6 &    92 & 126\\  
$\mathrm{H_2O}$& PQGT1
 &-1.2-2 & 2.4-6 & 2.5-7 &  1.3-5 &   124 & 408
 &-1.2-2 & 2.3-7 & 2.9-7 &  1.3-5 &   126 & 1070\\  
$\mathrm{H_2O}$& PQGT1T2
 &-2.9-3 & 3.3-7 & 2.9-7 &  1.5-5 &   316 & 6213
 &-2.9-3 & 2.6-6 & 2.9-7 &  1.7-5 &   289 & 19392\\  
$\mathrm{H_2O}$& PQGT1T2'
 &-2.0-3 & 1.2-6 & 3.0-7 &  1.1-5 &   257 & 4679
 &-2.0-3 & 2.4-6 & 2.9-7 &  9.0-6 &   252 & 16339\\  
\hline
\end{longtable}
}
\end{table}

In addition to identifying block diagonal structures in the N-representibility
constraints, we can further improve the efficiency of SSNSDPL and SSNSDPH 
by taking advantage of the low rank structure of the variable matrices. 
 Recall from Theorem \ref{thm:lowrank}
that the ratios of the rank  
of the $X_j$ matrix (denoted by $r_j$) associated with the T2 condition 
over the dimension of $X_j$ (denoted by $d_j$) should be bounded 
 by $(\frac{\sqrt{3}}{8}(d^2+6))/(\frac{d^2}{8}(\frac{3d}{2}-1))$. 
For the C atom and CH molecule, $d$ is 20 and 24 respectively.
Thus, at the solution the ratios should be bounded by 0.06 and 0.05, respectively. In Figure~\ref{fig:rank}, we replace $r_j$ by the numerical rank computed from
the eigenvalue decompositions of the $X_j$ variable and show the ratios  
 for $j$'s that are 
 associated with the four largest $d_j$'s at each DRS iteration. 
 We observe that these ratios can be relatively high in the first 
few iterations. But they eventually become less than 0.1 after
a few hundred iterations. 
  This property is useful \eqref{eq:mv1} in the DRS and  the semi-smooth Newton
  methods.
%since  only a few eigenvalues instead of all eigenvalues in the projection to the semidefinite cone are needed.
 It follows from \eqref{eq:drs-admm2} that  the $X$ variable is the projection
 of the $Z$ variable to
 semidefinite cone and  $|\alpha|$ in \eqref{eq:mv1} is equal to the rank of
 $X_j$ in the case of 2-RDM. Therefore, solving the Newton system  
\eqref{eq:nleqn2} becomes much cheaper by using \eqref{eq:mv1} when $|\alpha|$
 is small.

%of the $X_j$ matrix (denoted by $r_j$) associated with the T2 condition 
%over the dimension of $X_j$ (denoted by $d_j$) for $j$'s that are 
%associated with the four largest $d_j$'s at each ADMM iteration.  
%In Figure~\ref{fig:rank}, we use the C atom and CH molecule
%as examples to show the ratio between.  
%%\CY{why use ADMM?} 
%This property is useful in the ADMM and our semismooth Newton methods,
%since  only a few eigenvalues instead of all eigenvalues in the
%projection to the semidefinite cone are needed. The low rank property is also helpful for the efficiency of the semi-smooth Newton method.
%  %Table \ref{tab:blk}, 
%%\CY{In what way, need better description here. Why do we care about ADMM?}
% Because our semi-smooth Newton is based on the DRS and the equivalence between
% the DRS and ADMM,

\begin{figure}[!htb]
\centering
\subfigure[C]{
\includegraphics[width=0.45\textwidth,height=0.4\textwidth]{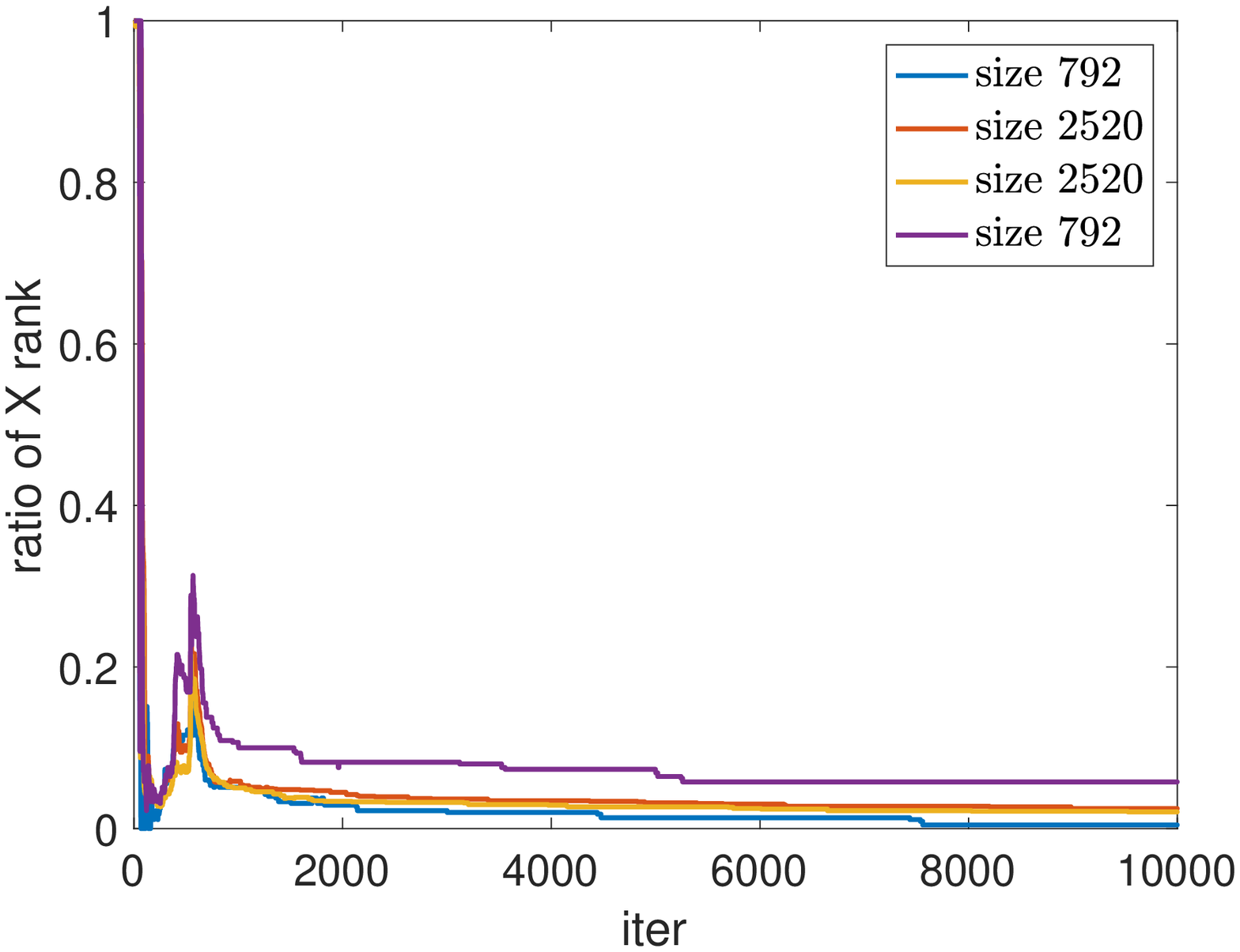}}
\subfigure[CH]{
\includegraphics[width=0.45\textwidth,height=0.4\textwidth]{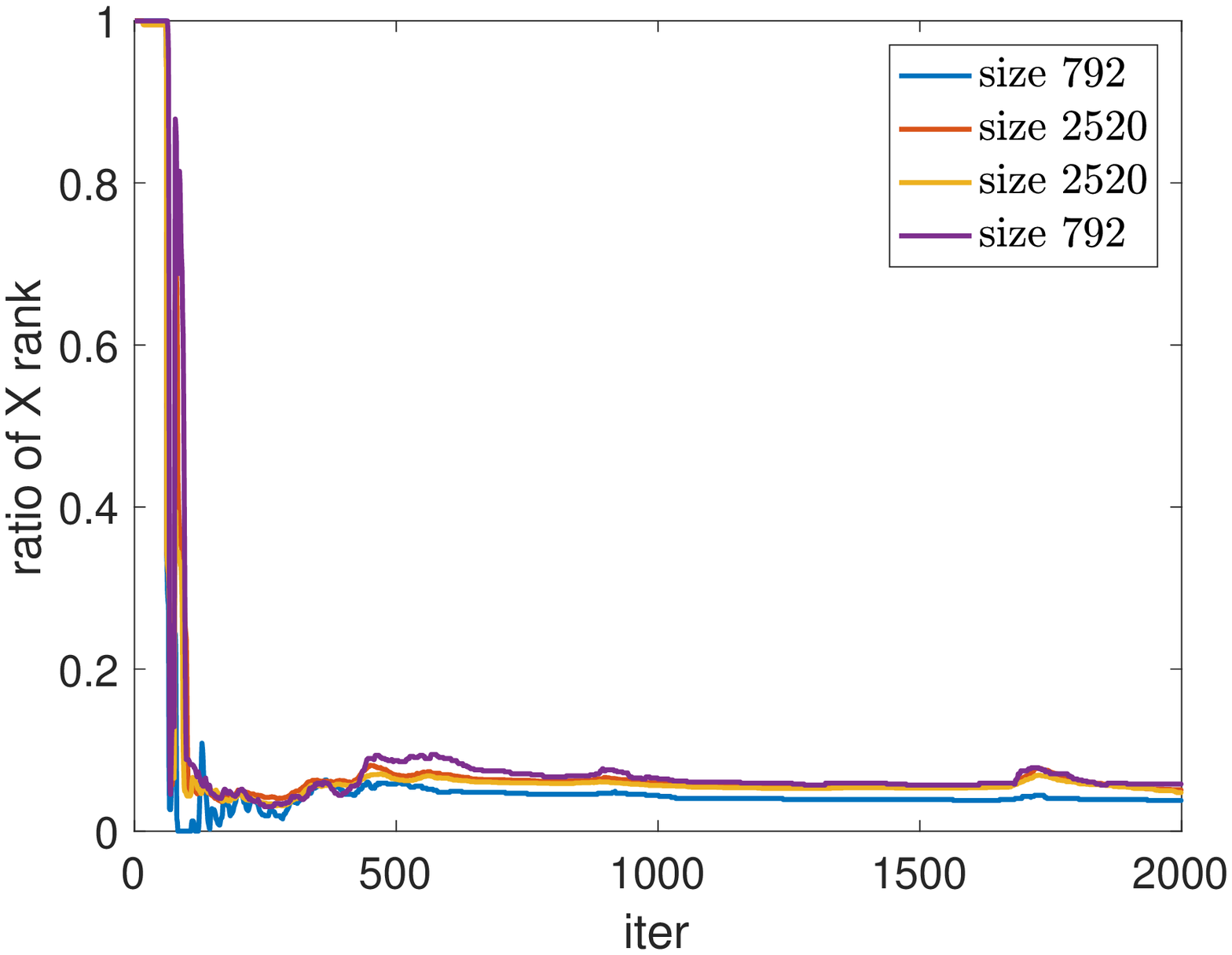}}
\caption{The percentage of the ranks of the first four largest blocks of the $X$
in the C
and CH systems}\label{fig:rank}
\end{figure}

Figure \ref{fig:gap} shows how the relative gap, primal infeasibility and dual
infeasibility in ADMM and SSNSDPL change with respect to the number of 
iterations when they are applied to the C atom system. 
We tested both algorithms on SDPs with the PQG N-representibility conditions 
(shown in subfigures a and c) and with the PQGT1T2 N-representibility 
conditions (shown in subfigures b and d.) In subfigures (a) and (b), 
we show the convergence history of ADMM for the first 10000 steps.  
In subfigures (c) and (d), we show the convergence history of SSNSDPL. 
The starting points of SSNSDPL are taken to be the solution produced
from running $500$ ADMM steps. %\CY{please check whether this is true.  Not sure whether it is necessary to show these.}
%Subfigures (e) and (f) provide a zoom-in view of the SSNSDPL convergence history. 
We can see that the ADMM can produce a moderately 
 accurate solution in a few hundred iterations from subfigures (a) and (b). At that point, 
convergence becomes slow.  Many more iterations (10,000) are required
to reach high accuracy.  Using a starting point obtained from running 500 
ADMM iterations, we can use SSNSDPL to obtain a more accurate solution
in $250$ steps. 
%Theydemonstrate that our semi-smooth Newton method indeed has an accelerated effect. However, the three
Note that the duality gap as well as the primal and dual feasibility 
curves shown in (c) and (d) are highly oscillatory. The oscillation
is due to the adaptive update of the penalty parameter $\mu$ 
for achieving a faster overall convergence rate. If the penalty parameter 
is held fixed, these curves become much smoother. But more iterations are
needed to reach the desired accuracy.

\begin{figure}[!htb]
\centering
\subfigure[ADMM on C(PQG)]{
\includegraphics[width=0.45\textwidth,height=0.4\textwidth]{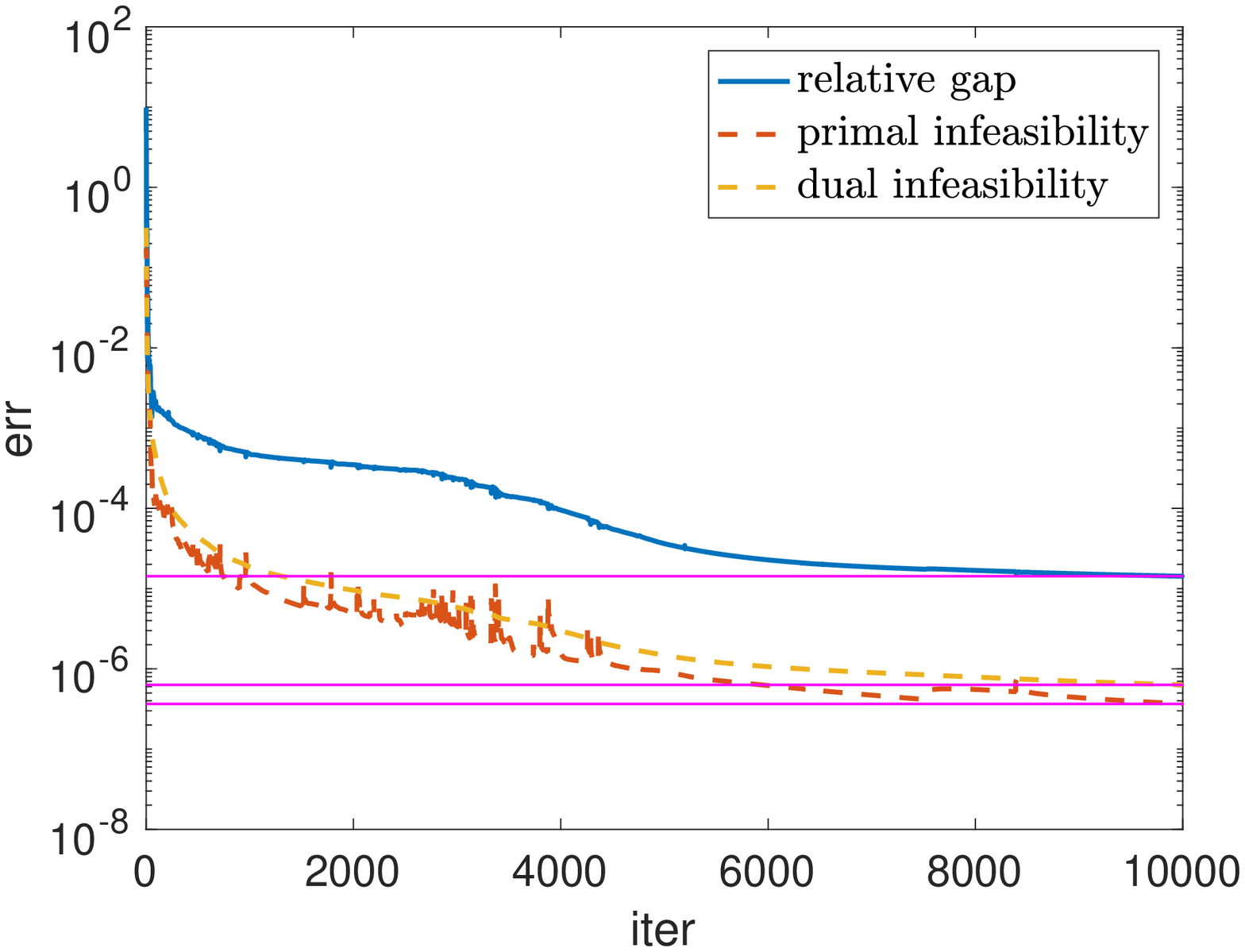}}
\vspace{5pt}
\subfigure[ADMM on C(PQGT1T2)]{
\includegraphics[width=0.45\textwidth,height=0.4\textwidth]{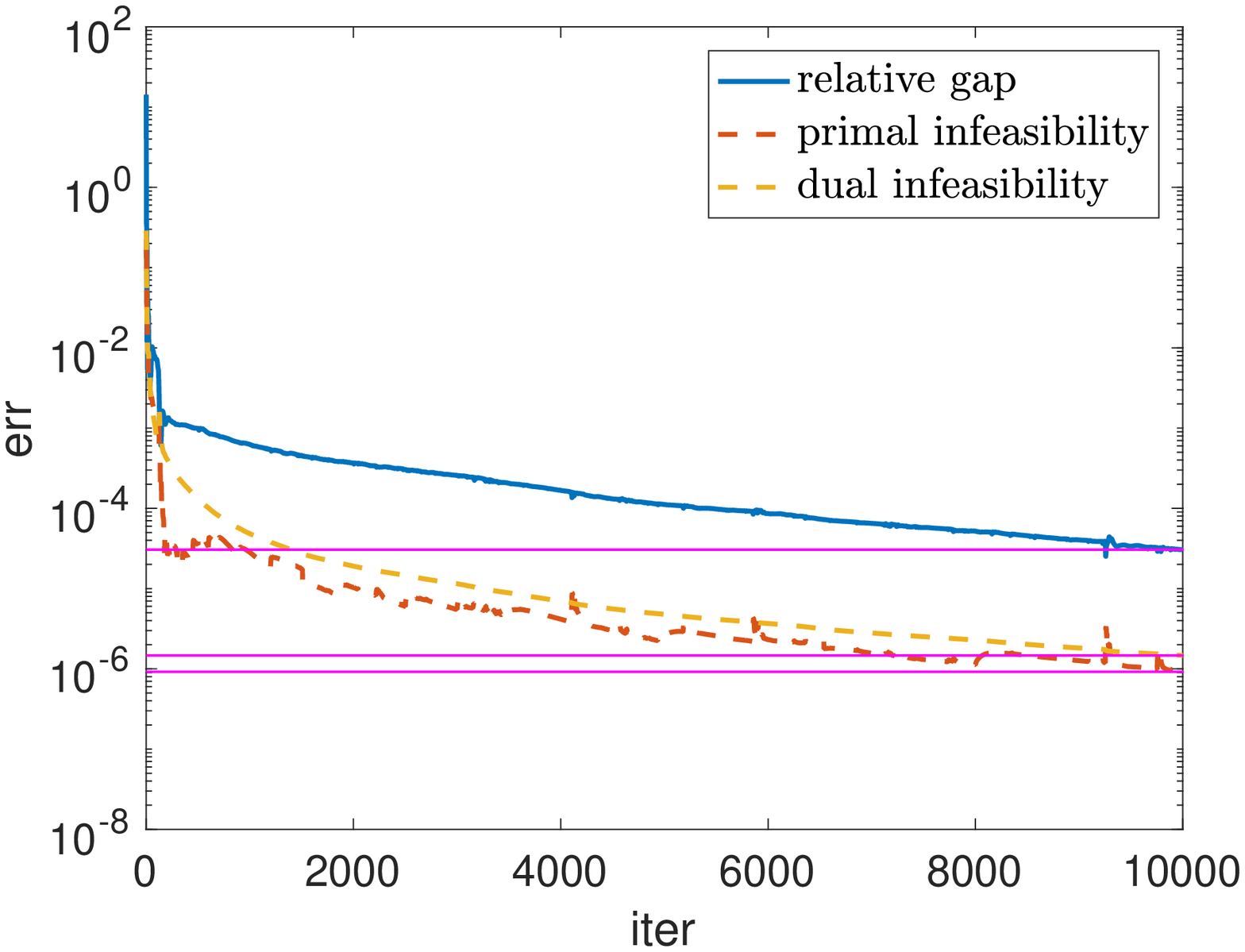}}
\subfigure[SSNSDPL on C(PQG)]{
%\includegraphics[width=0.45\textwidth,height=0.4\textwidth]{ssnsdp-trace73}}
%\subfigure[SSNSDPL on C(PQGT1T2)]{
%\includegraphics[width=0.45\textwidth,height=0.4\textwidth]{ssnsdp-trace75}}
%\subfigure[SSNSDPL (Newton) on C(PQG)]{
\includegraphics[width=0.45\textwidth,height=0.4\textwidth]{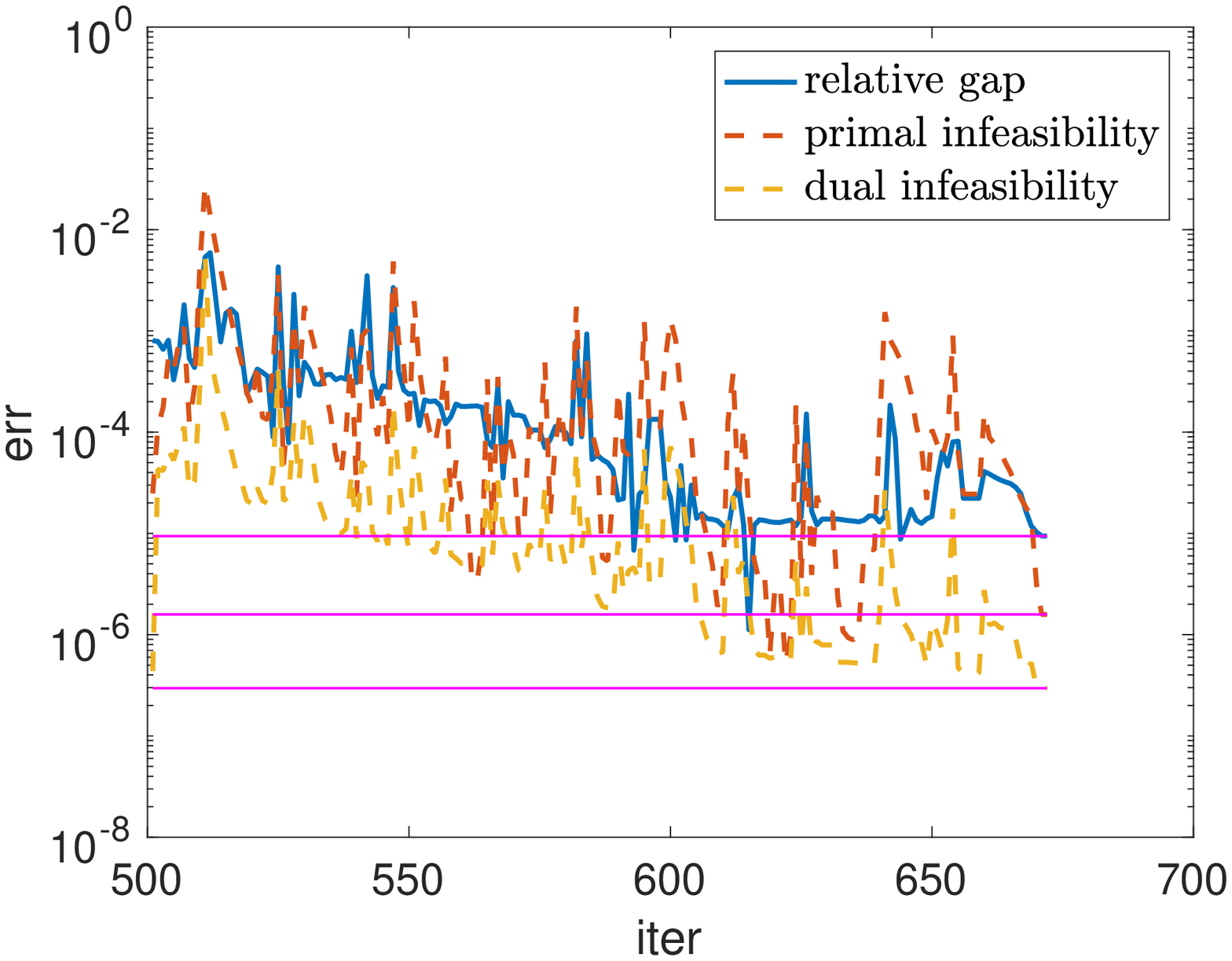}}
\subfigure[SSNSDPL on C(PQGT1T2)]{
\includegraphics[width=0.45\textwidth,height=0.4\textwidth]{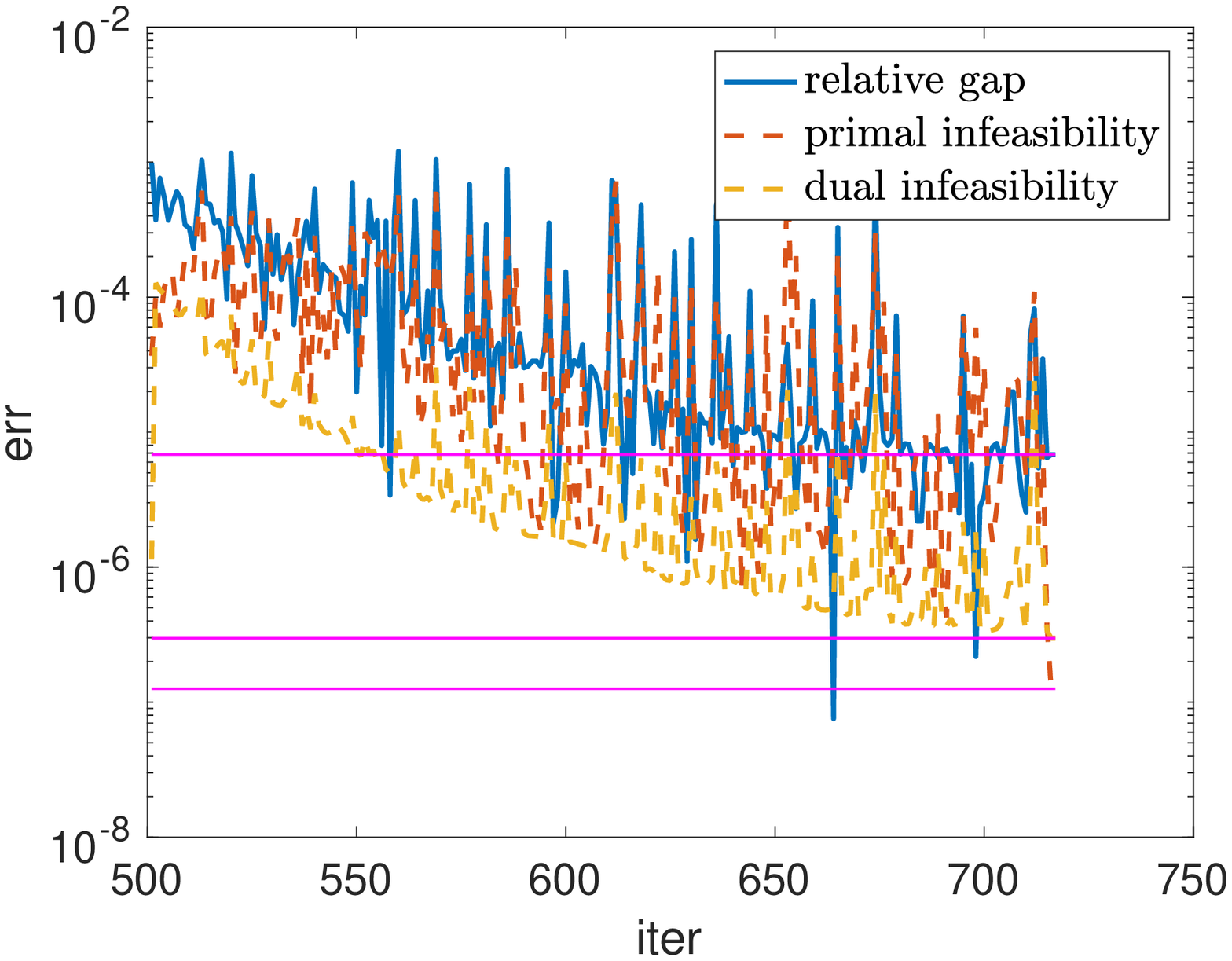}}
\caption{Relative gap, primal infeasibility and dual infeasibility}\label{fig:gap}
\end{figure}

The SSNSDPL and SSNSDPH methods have been successfully used to solve
the SDPs with several types of N-representability conditions for 
all test problems. In Table \ref{tab:err0}, we report the 
accuracy of the solution produced by SSNSDPH by comparing the 2-RDM ground 
state energy with the FCI energy and calculating their differences defined by
\eqref{eq:fcierr}. We can see that more accurate solutions are obtained
from SSNSDPH when more N-representability conditions are included in 
the constraints. These results are similar to the ones reported in
~\cite{Nakata2008}. %It illustrates the efficiency of our methods.
%\CY{should try to avoid using '-0.0000', which is confusing, use exponential form.}

%\begin{table}	
{
\footnotesize
\begin{longtable}{| c || c | c |  c |  c |  c |  c |}
\caption{The error obtained by SSNSDPH on various N-representability conditions: PQG, PQGT1, PQGT1T2, and PQGT1T2'}\label{tab:err0}\\
\hline 
system & state & basis & PQG & PQGT1 & PQGT1T2 & PQGT1T2' \\ 
\hline 
\endfirsthead 
\hline 
system & state & basis & PQG & PQGT1 & PQGT1T2 & PQGT1T2' \\ 
\hline 
\endhead 
\hline 
\endfoot 
$\mathrm{AlH}$ & 1Sigma+ & STO6G& -2.3-3& -7.8-4& -2.4-5& -1.4-5\\ 
$\mathrm{B_2}$ & 3Sigmag- & STO6G& -9.6-2& -8.5-2& -6.5-2& -6.4-2\\ 
$\mathrm{BF}$ & 1Sigma+ & STO6G& -6.6-3& -3.5-3& -3.2-4& -3.1-4\\ 
$\mathrm{BH^+}$ & 2Sigma+ & STO6G& -4.2-5& -2.6-5& -1.0-6& -2.9-7\\ 
$\mathrm{BH}$ & 1Sigma+ & DZ& -6.5-3& -4.7-3& -8.6-5& -5.1-5\\ 
$\mathrm{BH_3O}$ & 1A1 & STO6G& -2.8-2& -1.2-2& -7.1-4& -6.9-4\\ 
$\mathrm{BN}$ & 3Pi & STO6G& -2.9-2& -1.7-2& -3.0-3& -2.7-3\\ 
$\mathrm{BO}$ & 2Sigma+ & STO6G& -1.2-2& -6.7-3& -1.2-3& -1.0-3\\ 
$\mathrm{Be(1)}$ & 1S & STO6G& -4.6-7& -4.8-7& -9.8-8& -1.2-7\\ 
$\mathrm{Be(2)}$ & 1S & SV& -5.8-5& -5.4-5& -1.8-6& -6.4-7\\ 
$\mathrm{BeF}$ & 2Sigma+ & STO6G& -3.1-3& -1.7-3& -2.6-4& -1.9-4\\ 
$\mathrm{BeH^+}$ & 1Sigma+ & STO6G& -2.4-5& -2.3-5& -2.5-7& -1.9-7\\ 
$\mathrm{BeH}$ & 2Sigma+ & STO6G& -4.5-5& -2.2-5& -5.3-7& -2.5-7\\ 
$\mathrm{BeO}$ & 1Sigma+ & STO6G& -1.3-2& -9.5-3& -1.7-3& -1.7-3\\ 
$\mathrm{C(1)}$ & 3P & DZ& -3.9-3& -3.1-3& -3.9-4& -5.1-5\\ 
$\mathrm{C(2)}$ & 3PSZ0 & DZ& -1.7-2& -1.4-2& -2.4-3& -2.0-3\\ 
$\mathrm{C_2^-}$ & 2Sigmag+ & STO6G& -2.6-2& -1.4-2& -2.4-3& -1.9-3\\ 
$\mathrm{C_2(1)}$ & 1Sigmag+ & STO6G& -4.6-2& -2.5-2& -3.4-3& -3.5-3\\ 
$\mathrm{C_2(2)}$ & 1Sigmag+ & VDZ& -5.4-2& -5.4-2& -3.2-3& -3.5-3\\ 
$\mathrm{CF}$ & 2Pir & STO6G& -7.7-3& -5.8-3& -6.2-4& -4.8-4\\ 
$\mathrm{CH}$ & 2Pir & DZ& -1.3-2& -9.6-3& -8.9-4& -3.1-4\\ 
$\mathrm{CH_2(1)}$ & 1A1 & DZ& -1.9-2& -1.2-2& -3.9-4& -3.1-4\\ 
$\mathrm{CH_2(2)}$ & 3B1 & DZ& 4.1-1& 4.2-1& 4.3-1& 4.3-1\\ 
$\mathrm{CH_3^+}$ & 1Ep & STO6G& -1.3-2& -3.8-3& -1.7-4& -1.6-4\\ 
$\mathrm{CH_3}$ & 2A2pp & VDZ& -1.7-2& -1.0-2& -9.4-4& -3.1-4\\ 
$\mathrm{CH_3N}$ & 1A1 & STO6G& -3.9-2& -1.6-2& -1.0-3& -9.8-4\\ 
$\mathrm{CH_4}$ & 1A1 & STO6G& -1.9-2& -4.1-3& -1.9-4& -1.8-4\\ 
$\mathrm{CN}$ & 2Sigma+ & STO6G& -2.4-2& -1.2-2& -2.1-3& -1.7-3\\ 
$\mathrm{CO^+}$ & 2Sigma+ & STO6G& -1.8-2& -9.2-3& -1.7-3& -1.4-3\\ 
$\mathrm{CO}$ & 1Sigma+ & STO6G& -1.2-2& -7.2-3& -8.6-4& -8.6-4\\ 
$\mathrm{F^-}$ & 1S & DZ+d& -1.2-2& -7.6-3& -3.8-4& -2.7-4\\ 
$\mathrm{FH_2^+}$ & 1A1 & STO6G& -1.1-3& -5.1-4& -1.7-5& -1.5-5\\ 
$\mathrm{H_2O}$ & 1A1 & DZ& -1.9-2& -1.1-2& -4.9-4& -4.0-4\\ 
$\mathrm{H_3}$ & 2A1p & DZ& -7.7-4& -5.5-4& -1.6-6& -7.9-8\\ 
$\mathrm{HF}$ & 1Sigma+ & DZ& -1.2-2& -5.8-3& -3.5-4& -2.7-4\\ 
$\mathrm{HLi_2}$ & 2A1 & STO6G& -1.0-3& -6.6-4& -7.2-5& -1.0-5\\ 
$\mathrm{HN_2^+}$ & 1Sigma+ & STO6G& -2.5-2& -1.1-2& -1.5-3& -1.5-3\\ 
$\mathrm{HNO}$ & 1Ap & STO6G& -1.9-2& -1.4-2& -8.9-4& -9.0-4\\ 
$\mathrm{Li}$ & 2S & STO6G& -3.3-8& -1.8-8& -1.7-8& -4.2-9\\ 
$\mathrm{Li_2}$ & 1Sigmag+ & STO6G& -3.7-4& -2.9-4& -6.2-6& -4.3-6\\ 
$\mathrm{LiF}$ & 1Sigma+ & STO6G& -1.6-3& -1.3-3& -2.5-4& -2.4-4\\ 
$\mathrm{LiH(1)}$ & 1Sigma+ & DZ& -3.5-4& -2.0-4& -2.0-6& -6.7-7\\ 
$\mathrm{LiH(2)}$ & 1Sigma+ & STO6G& -3.4-5& -2.5-5& -1.6-7& -9.3-8\\ 
$\mathrm{LiOH}$ & 1Sigma+ & STO6G& -8.6-3& -4.0-3& -5.8-4& -5.7-4\\ 
$\mathrm{N}$ & 4S & DZ& -2.4-3& -9.0-4& -9.8-5& -1.1-5\\ 
$\mathrm{N_2^+}$ & 2Sigmag+ & STO6G& -3.1-2& -1.6-2& -2.6-3& -2.2-3\\ 
$\mathrm{N_2}$ & 1Sigmag+ & STO6G& -1.2-2& -8.8-3& -1.2-3& -1.2-3\\ 
$\mathrm{NH(1)}$ & 1Delta & DZ& -1.7-2& -1.3-2& -4.9-4& -4.5-4\\ 
$\mathrm{NH(2)}$ & 3Sigma- & DZ& -9.7-3& -5.2-3& -5.4-4& -1.4-4\\ 
$\mathrm{NH_2^-(1)}$ & 1A1 & DZ& -2.4-2& -1.5-2& -6.5-4& -5.7-4\\ 
$\mathrm{NH_2^-(2)}$ & 1A1 & STO6G& -2.0-3& -1.3-3& -2.2-5& -2.0-5\\ 
$\mathrm{NH_3^+}$ & 2A2pp & STO6G& -9.8-3& -1.8-3& -2.0-4& -1.1-4\\ 
$\mathrm{NH_3}$ & 1A1 & VDZ& -2.3-2& -1.4-2& -5.0-4& -4.7-4\\ 
$\mathrm{NH_4^+}$ & 1A1 & STO6G& -1.7-2& -4.2-3& -2.3-4& -2.2-4\\ 
$\mathrm{Na}$ & 2S & STO6G& -1.0-3& -4.9-4& -5.2-5& -3.9-5\\ 
$\mathrm{NaH}$ & 1Sigma+ & STO6G& -3.5-3& -1.6-3& -8.3-5& -7.4-5\\ 
$\mathrm{Ne}$ & 1S & DZ& -6.7-3& -2.7-3& -2.3-4& -1.5-4\\ 
$\mathrm{O(1)}$ & 1D & DZ& -1.9-2& -1.4-2& -1.3-3& -1.2-3\\ 
$\mathrm{O(2)}$ & 3P & DZ& -1.2-2& -6.3-3& -6.9-4& -2.4-4\\ 
$\mathrm{O(3)}$ & 3PSZ0 & DZ& -2.3-2& -1.9-2& -2.8-3& -1.6-3\\ 
$\mathrm{O_2^+}$ & 2Pig & STO6G& -1.7-2& -1.5-2& -2.4-3& -2.1-3\\ 
$\mathrm{P}$ & 4S & 631G& -8.3-4& -3.0-4& -6.4-5& -7.3-6\\ 
$\mathrm{SiH_4}$ & 1A1 & STO6G& -1.9-2& -3.6-3& -1.9-4& -1.6-4\\ 

\end{longtable}
}
%\end{table}

In Table~\ref{tab:opt}, we compare the accuracy and efficiency of 
SDPNAL, SSNSDPL and SSNSDPH. The fifth column labeled by ``$\mathrm{itr}$"  
gives the total number of Newton systems that was solved. 
%For SSNSDPL and SSNSDPH, it refers to 
%the number of Newton systems \ref{eq:nleqn2} that was solved,  %the iterations of Algorithm \ref{algo:newton},  % iterations of Newton steps {eq:nleqn2} \CY{not sure what is meant here?}, 
%while for SDPNAL, it refers to the total number Newton systems that was solved. 
Therefore, it is meaningful 
to compare these columns. The column labeled by t
gives the CPU time in seconds. %The stopping rules for all algorithms are listed in the beginning of this section. 
From the table, we can observe that SSNSDPL and SDPNAL achieve the same 
level of accuracy.  In terms of efficiency, SSNSDPL seems to be faster than 
SDPNAL for most examples. We ran SSNSDPH with a smaller
$\eta_d$ than SSNSDPL. Hence, it produces more accurate energy values. 
Table~\ref{tab:opt} shows that the errors of SSNSDPH is indeed smaller than 
SSNSDPL and they are similar to these in \cite{Nakata2008}. 
%\CY{I suggest to take out the 'err' column, since 'err' is mainly used to assess the accuracy of the model,i.e., the effectiveness of N-representability conditions.  That will allow us to add 'e' to all numbers.}
 % wzw: let us still keep err.

%\begin{scriptsize}
{
\setlength{\tabcolsep}{0.7pt}
\footnotesize
\begin{longtable}{|c| c|c|c|c|c|c | c|c|c|c|c|c | c|c|c|c|c|c |}
\caption{A summary of computational results of SDPNAL, SSNSDPL and SSNSDPH.}\label{tab:opt}\\
\hline 
\multicolumn{1}{|c|}{}
 & \multicolumn{6}{|c|}{SDPNAL}
 & \multicolumn{6}{|c|}{SSNSDPL}
 & \multicolumn{6}{|c|}{SSNSDPH}\\ 
\hline 
id
 & err & $\eta_p$ & $\eta_d$ & $\eta_g$ & itr & t 
 & err & $\eta_p$ & $\eta_d$ & $\eta_g$ & itr & t 
 & err & $\eta_p$ & $\eta_d$ & $\eta_g$ & itr & t \\ \hline 
\endfirsthead 
\hline 
\multicolumn{1}{|c|}{}
 & \multicolumn{6}{|c|}{SDPNAL}
 & \multicolumn{6}{|c|}{SSNSDPH}
 & \multicolumn{6}{|c|}{SSNSDPL}\\ 
\hline 
id
 & err & $\eta_p$ & $\eta_d$ & $\eta_g$ & itr & t 
 & err & $\eta_p$ & $\eta_d$ & $\eta_g$ & itr & t 
 & err & $\eta_p$ & $\eta_d$ & $\eta_g$ & itr & t \\ \hline 
\endhead 
\hline 
\endfoot 
$\mathrm{AlH}$
 &-5.3-4 & 4.8-6 & 5.1-7 &  5.5-6 &   155 &  411
 &-3.6-4 & 8.4-7 & 3.0-7 &  1.2-6 &    84 & 305
 &-1.4-5 & 1.4-5 & 7.5-10 &  9.8-6 &   155 & 504\\  
$\mathrm{B_2}$
 &-6.5-2 & 1.7-5 & 7.3-7 &  1.2-5 &   225 &  2260
 &-6.5-2 & 6.6-7 & 2.7-7 &  5.3-6 &   182 & 1938
 &-6.4-2 & 1.5-5 & 8.3-10 &  1.6-5 &   197 & 2152\\  
$\mathrm{BF}$
 &-7.9-4 & 7.8-6 & 6.0-7 &  1.3-5 &   175 &  466
 &-7.0-4 & 2.1-6 & 2.7-7 &  3.7-6 &   134 & 433
 &-3.1-4 & 1.4-5 & 9.7-10 &  1.4-5 &   185 & 603\\  
$\mathrm{BH^+}$
 &-1.2-4 & 2.4-6 & 7.1-7 &  2.8-6 &   192 &  86
 &-9.0-5 & 2.0-6 & 2.3-7 &  2.4-7 &   163 & 72
 &-2.9-7 & 1.0-5 & 9.8-10 &  4.8-6 &   245 & 102\\  
$\mathrm{BH}$
 &-6.1-4 & 8.3-5 & 7.0-7 &  1.2-4 &   252 &  2004
 &-5.2-4 & 7.2-7 & 2.9-7 &  1.3-5 &   258 & 2151
 &-5.1-5 & 3.9-5 & 9.1-10 &  1.2-4 &   234 & 2105\\  
$\mathrm{BH_3O}$
 &-1.7-3 & 1.1-5 & 6.9-7 &  1.7-5 &   183 &  4567
 &-1.6-3 & 1.3-6 & 2.9-7 &  1.7-6 &    99 & 3216
 &-6.9-4 & 7.4-6 & 9.7-10 &  1.1-5 &   205 & 5667\\  
$\mathrm{BN}$
 &-3.3-3 & 2.2-5 & 7.2-7 &  1.9-5 &   214 &  494
 &-3.2-3 & 1.6-6 & 2.8-7 &  4.2-6 &   108 & 387
 &-2.7-3 & 1.3-5 & 1.0-9 &  1.9-5 &   246 & 723\\  
$\mathrm{BO}$
 &-1.6-3 & 9.2-6 & 7.0-7 &  1.6-5 &   171 &  666
 &-1.5-3 & 2.9-6 & 2.9-7 &  1.8-6 &    88 & 490
 &-1.0-3 & 1.1-5 & 1.0-9 &  2.6-5 &   223 & 1069\\  
$\mathrm{Be(1)}$
 &-4.7-5 & 2.2-7 & 9.5-7 &  1.5-6 &   116 &  19
 &-3.9-5 & 1.3-7 & 3.0-7 &  1.0-6 &   195 & 49
 &-1.2-7 & 9.2-6 & 4.1-10 &  1.5-6 &   227 & 54\\  
$\mathrm{Be(2)}$
 &-1.6-4 & 7.4-5 & 7.1-7 &  6.0-6 &   221 &  249
 &-1.4-4 & 6.1-7 & 3.0-7 &  5.8-6 &   464 & 412
 &-6.4-7 & 1.3-5 & 9.9-10 &  3.3-7 &   313 & 327\\  
$\mathrm{BeF}$
 &-6.6-4 & 1.2-5 & 6.6-7 &  1.8-5 &   177 &  482
 &-5.2-4 & 1.2-6 & 3.0-7 &  1.0-6 &   179 & 553
 &-1.9-4 & 9.4-6 & 9.9-10 &  1.1-5 &   187 & 608\\  
$\mathrm{BeH^+}$
 &-7.3-5 & 9.1-6 & 5.6-7 &  4.5-6 &   198 &  95
 &-7.4-5 & 3.8-7 & 2.8-7 &  3.1-6 &   254 & 96
 &-1.9-7 & 1.2-5 & 9.9-10 &  1.7-6 &   217 & 87\\  
$\mathrm{BeH}$
 &-8.5-5 & 6.4-6 & 7.6-7 &  4.7-6 &   201 &  91
 &-7.1-5 & 1.7-6 & 3.0-7 &  1.8-8 &   144 & 65
 &-2.5-7 & 1.8-5 & 9.5-10 &  9.1-7 &   229 & 101\\  
$\mathrm{BeO}$
 &-2.2-3 & 1.3-5 & 7.2-7 &  2.3-5 &   199 &  495
 &-2.1-3 & 1.8-6 & 2.7-7 &  5.8-6 &   105 & 375
 &-1.7-3 & 6.4-6 & 9.7-10 &  1.4-5 &   220 & 695\\  
$\mathrm{C(1)}$
 &-5.5-4 & 3.3-5 & 5.9-7 &  3.2-5 &   245 &  440
 &-5.4-4 & 1.2-6 & 3.0-7 &  7.7-6 &   226 & 361
 &-5.1-5 & 1.3-5 & 8.5-10 &  3.7-5 &   295 & 428\\  
$\mathrm{C(2)}$
 &-2.6-3 & 1.5-5 & 6.5-7 &  1.3-5 &   233 &  424
 &-2.6-3 & 1.5-6 & 3.0-7 &  7.3-6 &   229 & 355
 &-2.0-3 & 1.2-5 & 9.2-10 &  2.4-5 &   230 & 386\\  
$\mathrm{C_2^-}$
 &-2.4-3 & 7.0-6 & 6.1-7 &  1.4-5 &   175 &  319
 &-2.4-3 & 9.9-7 & 2.9-7 &  3.9-6 &   115 & 244
 &-1.9-3 & 9.6-6 & 9.0-10 &  1.9-5 &   235 & 418\\  
$\mathrm{C_2(1)}$
 &-4.2-3 & 5.2-6 & 7.7-7 &  5.3-6 &   184 &  320
 &-4.1-3 & 9.6-7 & 3.0-7 &  3.3-6 &   112 & 241
 &-3.5-3 & 8.8-6 & 9.4-10 &  7.2-6 &   214 & 386\\  
$\mathrm{C_2(2)}$
 &3.7-3 & 5.2-4 & 7.1-7 &  5.3-4 &   234 &  7291
 &-5.1-3 & 1.6-6 & 2.8-7 &  7.0-6 &   217 & 8381
 &-3.5-3 & 1.4-5 & 9.4-10 &  3.2-5 &   154 & 6858\\  
$\mathrm{CF}$
 &-8.8-4 & 8.5-6 & 5.1-7 &  1.1-5 &   168 &  443
 &-8.1-4 & 1.4-6 & 2.9-7 &  2.1-6 &   115 & 437
 &-4.8-4 & 9.9-6 & 9.7-10 &  1.1-5 &   194 & 634\\  
$\mathrm{CH}$
 &-1.1-3 & 9.3-5 & 6.4-7 &  1.1-4 &   234 &  1875
 &-1.1-3 & 5.1-7 & 2.9-7 &  9.9-6 &   294 & 2041
 &-3.1-4 & 1.9-5 & 1.0-9 &  5.6-5 &   272 & 2262\\  
$\mathrm{CH_2(1)}$
 &-1.3-3 & 1.4-4 & 6.5-7 &  2.4-4 &   241 &  4834
 &-1.3-3 & 6.5-7 & 3.0-7 &  1.3-5 &   278 & 5870
 &-3.1-4 & 3.9-5 & 8.7-10 &  1.2-4 &   321 & 7671\\  
$\mathrm{CH_2(2)}$
 &4.3-1 & 6.4-4 & 6.5-7 &  2.0-4 &   251 &  4383
 &4.3-1 & 9.7-7 & 2.9-7 &  1.3-5 &   327 & 7595
 &4.3-1 & 5.0-5 & 9.6-10 &  7.9-5 &   375 & 9245\\  
$\mathrm{CH_3^+}$
 &-5.6-4 & 1.0-6 & 9.0-7 &  2.7-6 &   158 &  151
 &-4.5-4 & 6.7-7 & 2.3-7 &  1.7-6 &   135 & 133
 &-1.6-4 & 1.2-5 & 8.3-10 &  2.7-6 &   181 & 185\\  
$\mathrm{CH_3}$
 &-1.4-3 & 3.3-5 & 7.8-7 &  5.0-5 &   203 &  7744
 &-1.2-3 & 1.5-6 & 3.0-7 &  8.8-6 &   265 & 6474
 &-3.1-4 & 1.1-5 & 8.7-10 &  1.7-5 &   212 & 6423\\  
$\mathrm{CH_3N}$
 &-2.0-3 & 9.1-6 & 6.3-7 &  1.1-5 &   170 &  4100
 &-1.9-3 & 9.2-7 & 2.8-7 &  1.5-6 &   104 & 3160
 &-9.8-4 & 1.1-5 & 9.7-10 &  1.8-5 &   203 & 5250\\  
$\mathrm{CH_4}$
 &-7.5-4 & 9.3-7 & 8.9-7 &  3.7-6 &   148 &  164
 &-6.0-4 & 3.8-7 & 2.6-7 &  2.9-6 &   109 & 155
 &-1.8-4 & 8.5-6 & 9.8-10 &  1.1-5 &   168 & 239\\  
$\mathrm{CN}$
 &-2.2-3 & 1.2-5 & 5.4-7 &  2.0-5 &   185 &  461
 &-2.2-3 & 1.5-6 & 2.4-7 &  5.6-6 &   110 & 366
 &-1.7-3 & 9.1-6 & 1.0-9 &  1.9-5 &   277 & 724\\  
$\mathrm{CO^+}$
 &-2.0-3 & 1.1-5 & 7.8-7 &  1.8-5 &   174 &  435
 &-2.0-3 & 1.8-6 & 2.9-7 &  4.7-6 &   118 & 357
 &-1.4-3 & 1.2-5 & 9.7-10 &  2.3-5 &   269 & 728\\  
$\mathrm{CO}$
 &-1.3-3 & 1.3-5 & 6.8-7 &  1.9-5 &   162 &  408
 &-1.2-3 & 2.6-6 & 2.7-7 &  2.2-6 &    89 & 328
 &-8.6-4 & 9.7-6 & 9.8-10 &  1.7-5 &   204 & 638\\  
$\mathrm{F^-}$
 &-2.0-3 & 8.8-5 & 5.2-7 &  1.5-4 &   238 &  1429
 &-2.0-3 & 1.2-6 & 2.7-7 &  1.1-5 &   249 & 1359
 &-2.7-4 & 2.0-5 & 8.5-10 &  7.8-5 &   282 & 1576\\  
$\mathrm{FH_2^+}$
 &-2.3-4 & 1.2-6 & 6.4-7 &  2.2-6 &   146 &  99
 &-1.8-4 & 1.5-6 & 2.9-7 &  1.5-6 &    55 & 53
 &-1.5-5 & 1.1-5 & 9.9-10 &  2.0-6 &   178 & 153\\  
$\mathrm{H_2O}$
 &-1.9-3 & 7.4-5 & 4.9-7 &  1.1-4 &   246 &  5704
 &-2.0-3 & 1.2-6 & 3.0-7 &  1.1-5 &   257 & 4679
 &-4.0-4 & 1.3-5 & 9.4-10 &  4.1-5 &   282 & 5928\\  
$\mathrm{H_3}$
 &-3.3-5 & 8.5-7 & 9.7-7 &  7.5-6 &   143 &  51
 &-2.0-5 & 3.0-7 & 3.0-7 &  4.2-6 &   176 & 58
 &-7.9-8 & 5.7-6 & 9.8-10 &  8.4-6 &   204 & 80\\  
$\mathrm{HF}$
 &-2.3-3 & 5.6-5 & 6.9-7 &  7.5-5 &   216 &  1745
 &-2.0-3 & 8.5-7 & 2.9-7 &  1.2-5 &   187 & 1438
 &-2.7-4 & 1.2-5 & 8.5-10 &  3.9-5 &   265 & 2038\\  
$\mathrm{HLi_2}$
 &-2.8-4 & 1.9-5 & 7.7-7 &  2.2-5 &   240 &  1108
 &-1.9-4 & 2.0-6 & 2.9-7 &  6.9-6 &   447 & 1357
 &-1.0-5 & 3.0-5 & 8.8-10 &  2.3-5 &   168 & 755\\  
$\mathrm{HN_2^+}$
 &-2.2-3 & 8.4-6 & 7.8-7 &  1.1-5 &   167 &  720
 &-2.0-3 & 2.2-6 & 2.9-7 &  2.5-6 &    88 & 530
 &-1.5-3 & 1.4-5 & 9.9-10 &  1.8-5 &   232 & 1108\\  
$\mathrm{HNO}$
 &-1.5-3 & 1.4-5 & 7.1-7 &  2.3-5 &   193 &  1551
 &-1.3-3 & 1.0-6 & 2.0-7 &  3.7-6 &   125 & 1300
 &-9.0-4 & 1.0-5 & 9.9-10 &  5.8-6 &   238 & 2430\\  
$\mathrm{Li}$
 &-1.7-5 & 2.1-7 & 6.8-7 &  1.8-6 &   125 &  23
 &-1.2-5 & 1.5-6 & 2.4-7 &  1.1-6 &   145 & 34
 &-4.2-9 & 2.0-5 & 5.1-10 &  1.8-6 &   153 & 32\\  
$\mathrm{Li_2}$
 &-2.0-4 & 2.5-5 & 6.9-7 &  2.5-5 &   242 &  502
 &-1.6-4 & 1.6-6 & 2.9-7 &  5.7-6 &   418 & 625
 &-4.3-6 & 3.3-5 & 9.4-10 &  2.9-5 &   183 & 326\\  
$\mathrm{LiF}$
 &-6.6-4 & 9.6-6 & 6.2-7 &  1.1-5 &   197 &  535
 &-5.6-4 & 2.3-6 & 2.6-7 &  2.4-6 &   103 & 381
 &-2.4-4 & 9.5-6 & 9.6-10 &  8.3-6 &   178 & 638\\  
$\mathrm{LiH(1)}$
 &-1.2-4 & 2.7-5 & 7.4-7 &  1.8-5 &   233 &  1774
 &-8.9-5 & 1.6-6 & 2.9-7 &  6.7-6 &   464 & 2781
 &-6.7-7 & 1.6-5 & 9.1-10 &  2.4-5 &   265 & 2434\\  
$\mathrm{LiH(2)}$
 &-5.9-5 & 8.5-6 & 6.9-7 &  4.9-6 &   212 &  107
 &-5.2-5 & 1.6-6 & 2.9-7 &  7.0-6 &   256 & 101
 &-9.3-8 & 1.9-5 & 9.8-10 &  5.2-6 &   198 & 78\\  
$\mathrm{LiOH}$
 &-1.0-3 & 1.0-5 & 5.4-7 &  1.5-5 &   183 &  854
 &-9.7-4 & 1.3-6 & 3.0-7 &  2.0-6 &   107 & 630
 &-5.7-4 & 9.8-6 & 9.0-10 &  1.1-5 &   247 & 1253\\  
$\mathrm{N}$
 &-5.0-4 & 6.8-5 & 5.0-7 &  6.6-5 &   209 &  351
 &-4.6-4 & 2.3-6 & 3.0-7 &  7.6-6 &   229 & 384
 &-1.1-5 & 1.5-5 & 9.7-10 &  6.1-5 &   297 & 454\\  
$\mathrm{N_2^+}$
 &-2.8-3 & 5.6-6 & 7.6-7 &  1.1-5 &   167 &  300
 &-2.7-3 & 7.8-7 & 2.9-7 &  1.2-6 &   102 & 236
 &-2.2-3 & 8.7-6 & 9.8-10 &  1.7-5 &   263 & 496\\  
$\mathrm{N_2}$
 &-1.5-3 & 8.6-6 & 4.4-7 &  8.2-6 &   160 &  281
 &-1.5-3 & 1.5-6 & 2.6-7 &  2.4-7 &    96 & 214
 &-1.2-3 & 1.0-5 & 8.9-10 &  2.7-5 &   235 & 425\\  
$\mathrm{NH(1)}$
 &-1.3-3 & 4.5-5 & 5.1-7 &  7.3-5 &   244 &  2014
 &-1.3-3 & 2.8-7 & 2.8-7 &  7.6-6 &   291 & 1959
 &-4.5-4 & 1.6-5 & 9.9-10 &  3.8-5 &   230 & 1803\\  
$\mathrm{NH(2)}$
 &-9.7-4 & 1.1-4 & 5.2-7 &  1.6-4 &   233 &  1764
 &-1.0-3 & 1.3-6 & 3.0-7 &  7.3-6 &   272 & 1986
 &-1.4-4 & 1.3-5 & 9.1-10 &  3.9-5 &   256 & 2066\\  
$\mathrm{NH_2^-(1)}$
 &-1.8-3 & 7.0-5 & 5.0-7 &  1.3-4 &   235 &  5430
 &-1.7-3 & 1.3-6 & 2.7-7 &  8.5-6 &   253 & 4772
 &-5.7-4 & 1.2-5 & 9.6-10 &  4.5-5 &   258 & 5775\\  
$\mathrm{NH_2^-(2)}$
 &-1.6-4 & 1.5-6 & 4.8-7 &  1.8-6 &   151 &  96
 &-1.6-4 & 2.9-7 & 2.7-7 &  1.0-6 &    78 & 61
 &-2.0-5 & 4.9-6 & 8.2-10 &  3.8-6 &   211 & 145\\  
$\mathrm{NH_3^+}$
 &-3.7-4 & 1.3-6 & 5.6-7 &  1.7-6 &   175 &  179
 &-3.4-4 & 2.5-6 & 2.9-7 &  1.1-6 &   105 & 117
 &-1.1-4 & 9.9-6 & 9.6-10 &  5.0-6 &   222 & 222\\  
$\mathrm{NH_3}$
 &-1.6-3 & 9.6-6 & 5.8-7 &  1.7-5 &   239 &  13131
 &-1.6-3 & 1.4-7 & 2.9-7 &  5.8-6 &   227 & 10022
 &-4.7-4 & 1.2-5 & 9.7-10 &  2.0-5 &   217 & 10903\\  
$\mathrm{NH_4^+}$
 &-6.1-4 & 1.9-6 & 6.3-7 &  1.8-6 &   162 &  187
 &-5.1-4 & 1.6-6 & 1.9-7 &  9.9-7 &   109 & 160
 &-2.2-4 & 6.4-6 & 7.6-10 &  2.1-6 &   187 & 266\\  
$\mathrm{Na}$
 &-5.2-4 & 4.4-6 & 6.4-7 &  3.9-6 &   164 &  163
 &-3.5-4 & 6.6-7 & 2.2-7 &  1.1-6 &    92 & 127
 &-3.9-5 & 4.5-6 & 8.8-10 &  6.9-6 &   212 & 248\\  
$\mathrm{NaH}$
 &-7.9-4 & 5.4-6 & 7.2-7 &  6.2-6 &   179 &  485
 &-6.7-4 & 1.9-6 & 3.0-7 &  4.6-6 &   107 & 332
 &-7.4-5 & 9.0-6 & 1.0-9 &  9.1-6 &   199 & 604\\  
$\mathrm{Ne}$
 &-2.5-3 & 2.0-5 & 7.7-7 &  3.4-5 &   188 &  328
 &-1.8-3 & 2.9-6 & 3.0-7 &  6.8-6 &   140 & 246
 &-1.5-4 & 1.5-5 & 9.9-10 &  4.1-5 &   223 & 370\\  
$\mathrm{O(1)}$
 &-2.0-3 & 2.1-5 & 4.5-7 &  2.9-5 &   196 &  334
 &-2.0-3 & 1.8-6 & 2.7-7 &  5.3-6 &   228 & 339
 &-1.2-3 & 1.5-5 & 8.9-10 &  2.5-5 &   223 & 358\\  
$\mathrm{O(2)}$
 &-1.2-3 & 7.4-5 & 5.6-7 &  9.1-5 &   197 &  335
 &-1.2-3 & 9.8-7 & 2.8-7 &  7.0-6 &   206 & 325
 &-2.4-4 & 9.4-6 & 9.6-10 &  2.1-5 &   255 & 389\\  
$\mathrm{O(3)}$
 &-2.5-3 & 1.8-5 & 5.3-7 &  2.0-5 &   215 &  353
 &-2.5-3 & 6.0-7 & 3.0-7 &  6.1-6 &   191 & 309
 &-1.6-3 & 2.5-5 & 7.1-10 &  2.4-5 &   223 & 324\\  
$\mathrm{O_2^+}$
 &-2.4-3 & 4.4-6 & 5.6-7 &  6.5-6 &   152 &  284
 &-2.4-3 & 1.8-6 & 3.0-7 &  1.2-6 &   102 & 233
 &-2.1-3 & 7.7-6 & 9.9-10 &  5.6-6 &   201 & 462\\  
$\mathrm{P}$
 &-1.1-3 & 7.0-6 & 6.3-7 &  7.0-6 &   188 &  1149
 &-7.7-4 & 1.3-6 & 2.9-7 &  5.6-7 &   130 & 1017
 &-7.3-6 & 2.1-5 & 8.9-10 &  1.3-5 &   182 & 1254\\  
$\mathrm{SiH_4}$
 &-1.0-3 & 5.6-6 & 5.1-7 &  4.6-6 &   165 &  1755
 &-6.3-4 & 2.2-6 & 2.7-7 &  3.8-7 &    90 & 1256
 &-1.6-4 & 1.7-5 & 9.4-10 &  1.0-5 &   147 & 1961\\  
\end{longtable}
}
%\end{scriptsize}

 Finally, we compare the accuracy and efficiency of 
SSNSDP with that of SDPNAL using the the performance profiling method proposed 
in~\cite{dolan2002benchmarking}. 
Let  $t_{p,s}$ be the number of iterations or CPU time required to solve
problem $p$ by  the $s$th solvers. 
 Then one computes the ratio $r_{p,s}$  between $t_{p,s}$ over the smallest value obtained by $n_s$ solvers on problem $p$, i.e.,
 $r_{p,s}:=\frac{ t_{p,s}} {\min \{ t_{p,s}: 1\le s\le n_s\} }$. % whose value is set to some sufficiently large number whenever solver $s$ fails on problem $p$. Then, 
 For $ \tau \ge 0$, the value 
 \[ \pi_s(\tau) := \frac{\mbox{ number of problems where } \log_2(r_{p,s}) \le \tau
 } { \mbox{ total number of problems }}\] 
 indicates that solver $s$ is within a
factor $2^\tau\ge 1$ of the performance obtained by the best solver. Then the
performance plot is a curve $\pi_s(\tau)$ for each solver $s$ as a function of $\tau$.
 In Figure \ref{fig:perf}, we show the performance profiles of four criteria $\opt$, $\eta_d$, $\mathrm{err}$ and CPU time, where
$\opt = \max\{\eta_p, \eta_d, \eta_g \}$  represents the the largest value among
three optimal indexes $\eta_p$, $\eta_d$ and $\eta_g$. The  dual infeasibility $\eta_d$ is
chosen since it is often the smallest one among $\eta_p$, $\eta_d$ and $\eta_g$
for both SDPNAL and SSNSDPL.  These figures show that the accuracy and the CPU time of SSNSDPL are better than  SDPNAL on most test problems.

%\CY{need better explanation here.  I suggest dropping (c) since it doesn't convey much, and we already mentioned that SSNSDP gives more accurate results.}
% wzw: we still show (c): 1) the version SSNSDPL and SDPNAL is comparable in
% terms of err; 2) it does not good by only showing three figures

\begin{figure}[!htb]
\centering
\subfigure[opt]{
\includegraphics[width=0.45\textwidth,height=0.4\textwidth]{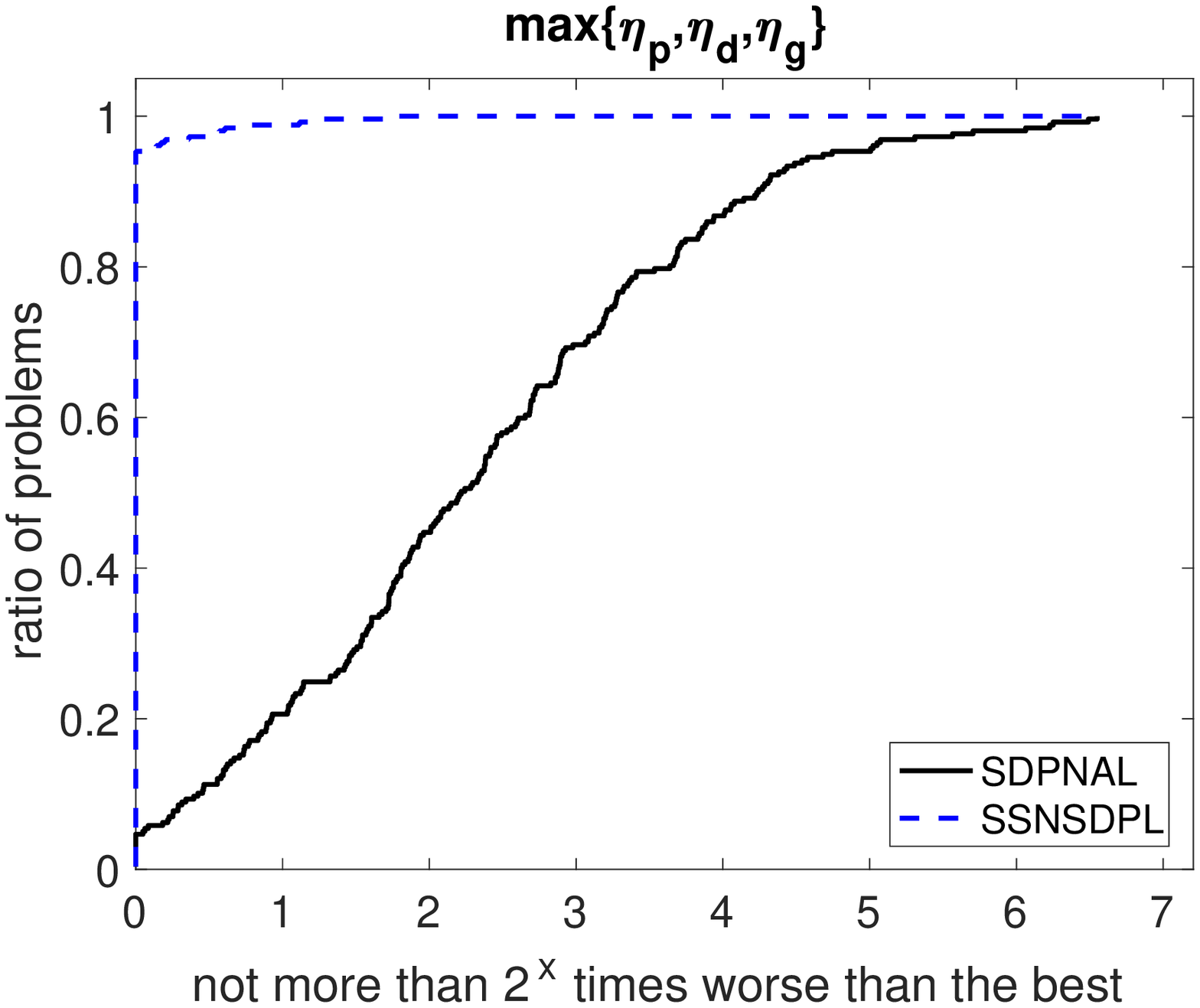}}
\subfigure[$\eta_d$]{
\includegraphics[width=0.45\textwidth,height=0.4\textwidth]{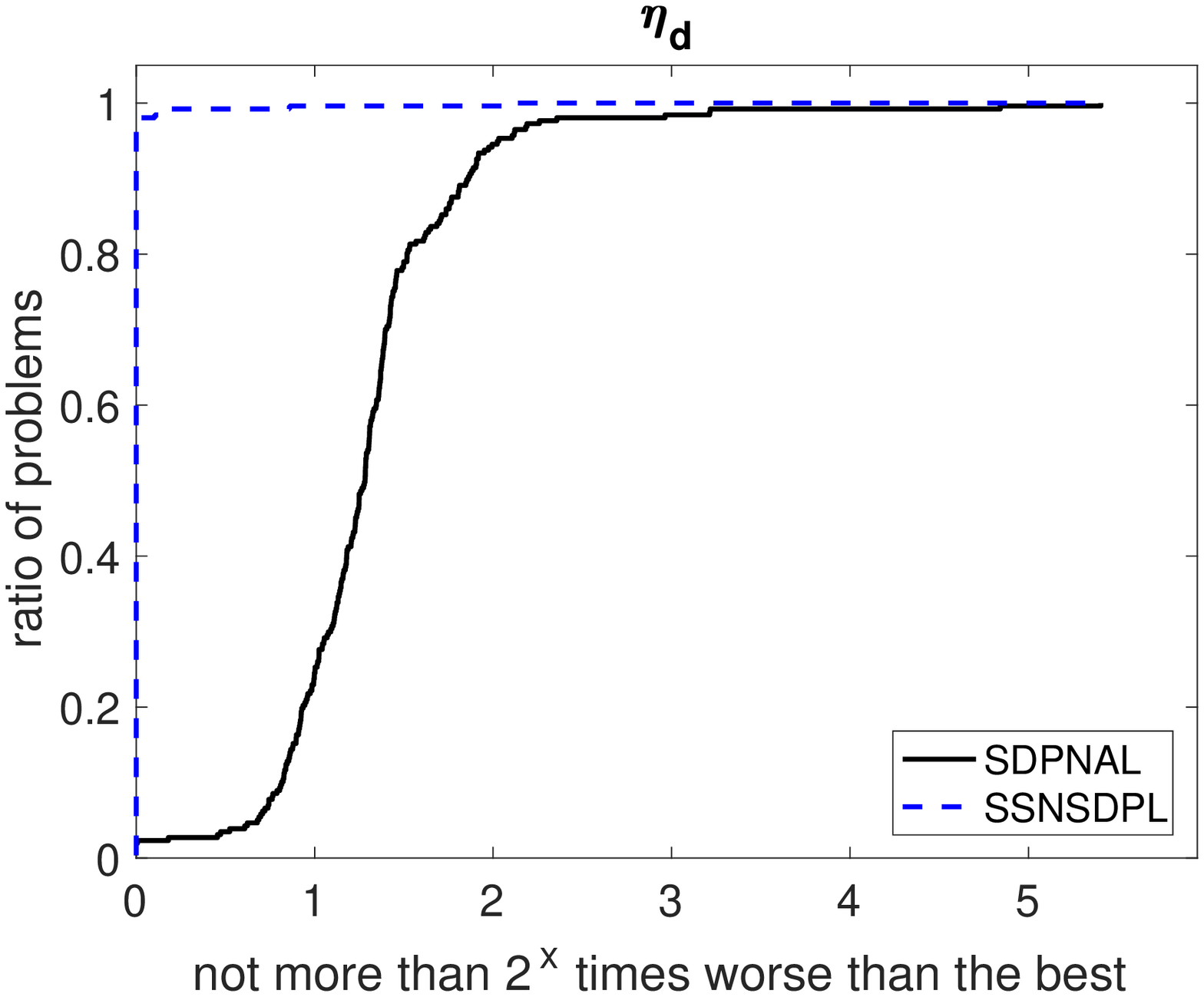}}
\subfigure[error]{
\includegraphics[width=0.45\textwidth,height=0.4\textwidth]{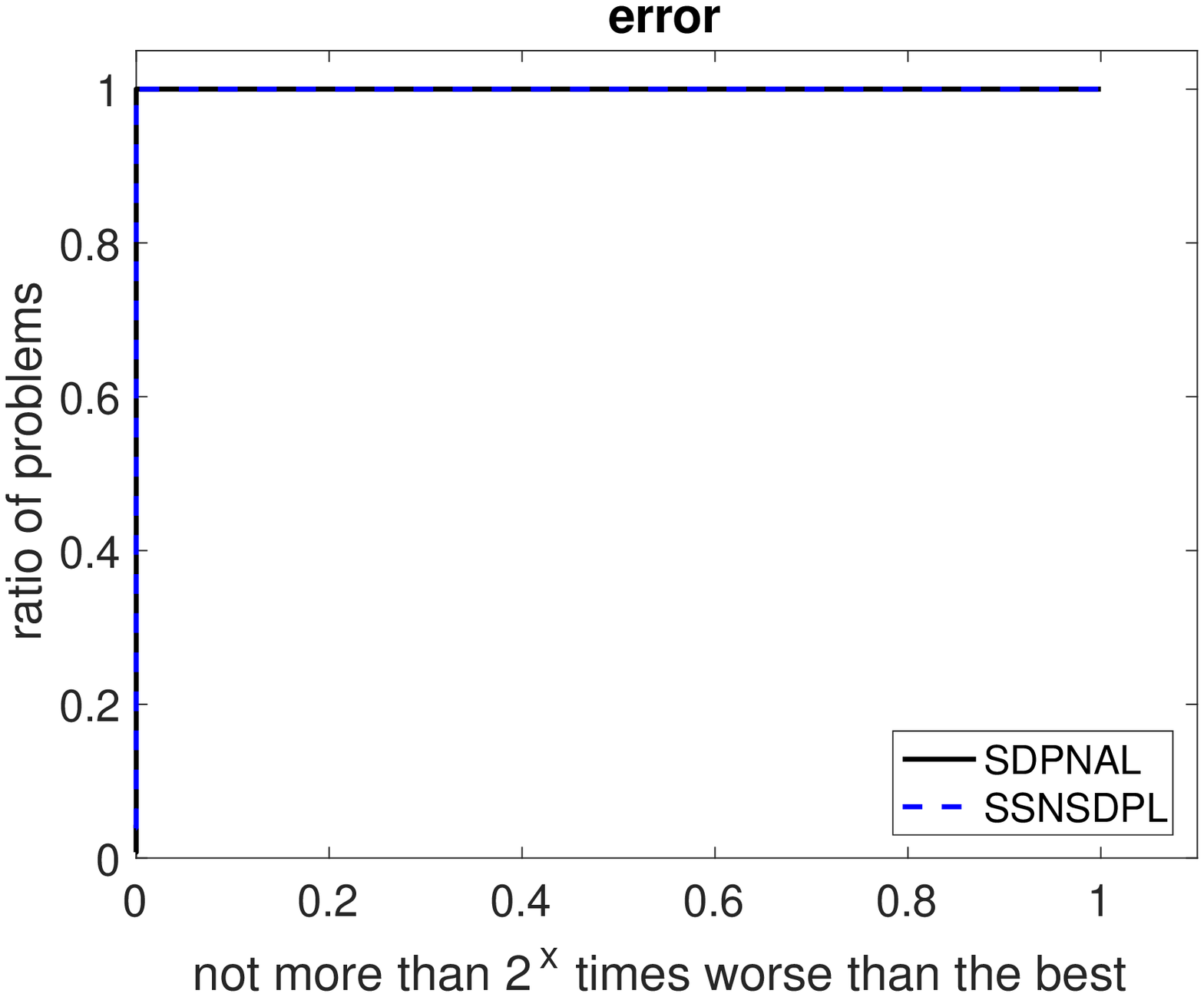}}
\subfigure[CPU]{
\includegraphics[width=0.45\textwidth,height=0.4\textwidth]{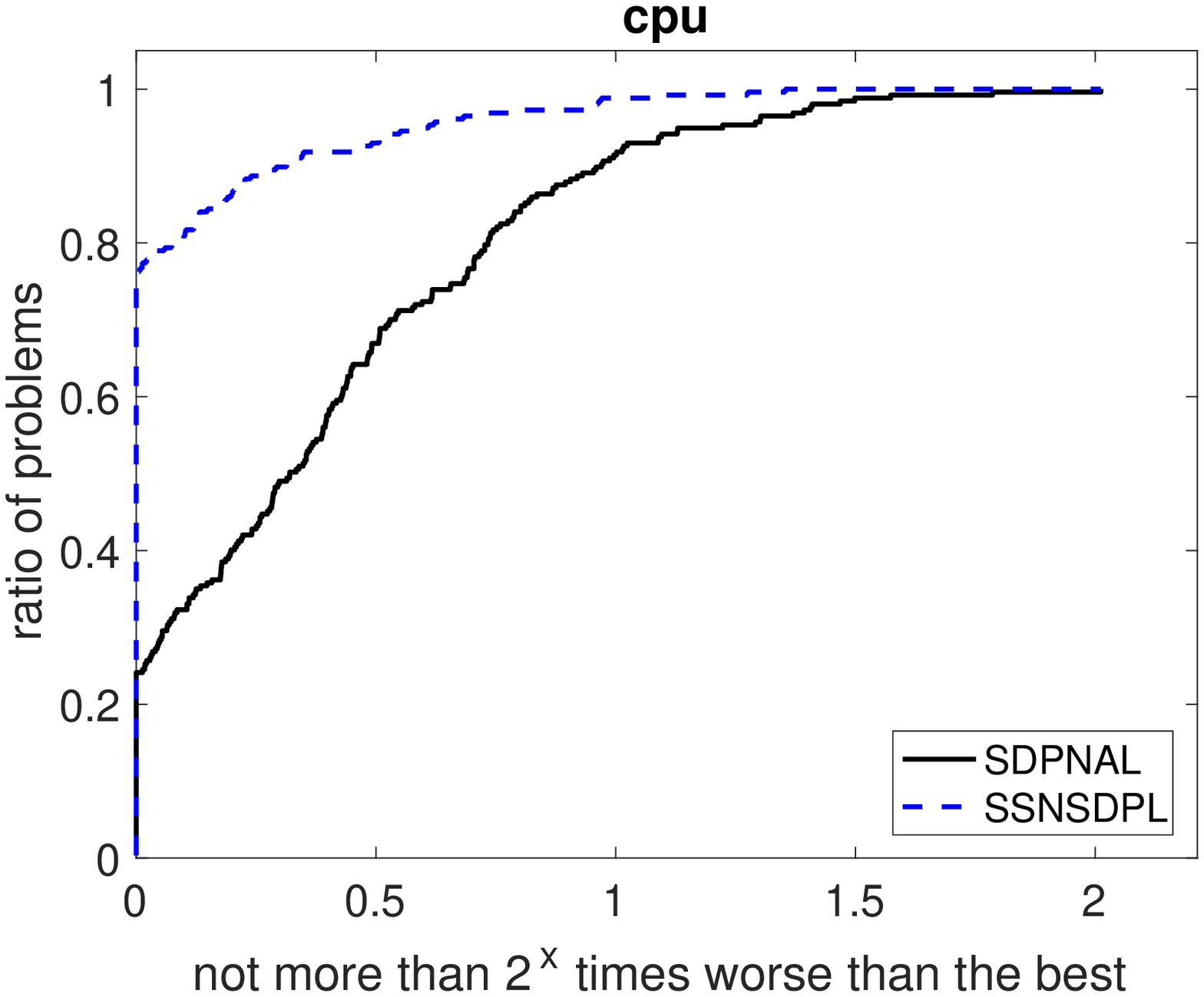}}
\caption{The performance profiles of SDPNAL and SSNSDPL}\label{fig:perf}
\end{figure}

\section{Conclusion}\label{sec:con}
In this paper, we consider the v2-RDM model for approximating the solution to the molecular Schr\"{o}dinger equation. Instead of computing the smallest eigenvalue of the many-electron Schr\"{o}dinger operator, we minimize the total energy of the many-electron system with respect to 1-RDM and 2-RDM subject to some linear constraints imposed to enhance the $N$-representability of the decision variables. The minimization problem to be solved is an SDP. The solution of the SDP can be obtained from the solution of a system of nonlinear equations that can be derived from a fixed point iteration of DRS applied to the original SDP.  We present a semi-smooth Newton type method for solving this set of nonlinear equations. A hyperplane projection technique is applied to improve the stability of the method and achieve global convergence.  We exploit the block diagonal structure and low rank structure of the variables in the SDP to improve the computational efficiency. The computational results show that the proposed semi-smooth Newton method can achieve higher accuracy, and is competitive with the Newton-CG Augmented Lagrangian Method for solving SDPs.

Several components of the proposed semi-smooth Newton method 
can be further improved. For example, since eigenvalue
decomposition is the most expensive step in the procedure 
for computing the Newton direction, a more efficient eigen-decomposition
methods needs to be investigated.  A better global convergent technique is  
also needed to improve the overall performance.

%%%%%%%%%%%%%%%%%%%%%%%%%%%%%%%%%%%%%%%%

\section*{Acknowledgments} 
The authors are grateful to Prof. Nakata Maho and Prof. Mituhiro Fukuta for
sharing all data sets on 2-RDM. 
%This work was partly supported by
%the Scientific Discovery through Advanced Computing (SciDAC) program 
%funded by U.S.~Department of Energy, Office of Science, Advanced Scientific 
%Computing Research and Basic Energy Sciences (C. Y.), and by the
%Center for Applied Mathematics for Energy Research Applications
%(CAMERA) (C. Y.). 
We also thank Jinmei Zhang for helping with test problem 
preparation. 

%%%%%%%%%%%%%%%%%%%%%%%%%%%%%%%%%%%%%%%%
\bibliographystyle{siam}
\bibliography{rdm.bbl}
%%%%%%%%%%%%%%%%%%%%%%%%%%%%%%%%%%%%%%%%%

\end{document}